\documentclass[12pt, a4paper]{amsart}

\usepackage{amsmath,amssymb,amsfonts, bm}
\usepackage{bbold}
\usepackage{enumerate}

\usepackage{amsthm}
\usepackage[all]{xy}
\usepackage{caption}
\usepackage{stmaryrd} 
\usepackage{tikz}
\usetikzlibrary{trees}
\usetikzlibrary{calc}

\pdfoutput=1 

\usepackage{geometry}
\usepackage{bussproofs}

\usepackage{listings} 

\theoremstyle{plain}
\newtheorem{thm}{Theorem}[section]
\newtheorem{prop}[thm]{Proposition}

\newtheorem{cor}[thm]{Corollary}

\theoremstyle{definition}
\newtheorem{defn}[thm]{Definition}
\newtheorem{exmp}[thm]{Example}
\newtheorem{Note}[thm]{Note}

\numberwithin{figure}{section}

\numberwithin{table}{section}

\DeclareMathOperator*{\esssup}{ess\,sup}

\newcommand{\lspace} {
  \vspace{0.8\baselineskip}
}

\newcommand{\abs}[1]{
  \lvert  #1 \rvert
}

\newcommand{\norm}[1]{
  \|  #1 \|
}

\newcommand{\twc}[2]{ 
  [ #1 ]_{ #2 }
}

\newcommand{\var}{\mathrm{Var}}

\newdir{ >}{{}*!/-5pt/@{>}}

\definecolor{arrowred}{rgb}{0,0,0} 
\newcommand{\newword}[1]{\textbf{\textit{#1}}}

\newcommand{\textred}[1]{#1}
\newcommand{\textblue}[1]{#1}

\numberwithin{equation}{section}

\newcommand{\Prob}{\mathbf{Prob}}

\newcommand{\Set}{\mathbf{Set}}
\newcommand{\Cat}{\mathbf{Cat}}
\newcommand{\Ord}{\mathbf{Ord}}

\newcommand{\Ban}{\mathbf{Ban}} 
\newcommand{\Top}{\mathbf{Top}}
\newcommand{\Mble}{\mathbf{Mble}}

\newcommand{\im}{\mathrm{im}}

\newcommand{\bDelta}{\mathbf{\Delta}}
\newcommand{\bSigma}{\mathbf{\Sigma}}

\mathchardef\mhyphen="2D  

\title{A geometric model of synthetic filtrations via context-dependent time}

\thanks{This work was supported by JSPS KAKENHI Grant Number 24K04941.} 

\author[T. Adachi]{Takanori Adachi}

\address{Graduate School of Management,
         Tokyo Metropolitan University,
         Marunouchi Eiraku Bldg. 18F, 1-4-4 Marunouchi, Chiyoda-ku, Tokyo 100-0005 Japan}
\email{Takanori Adachi <taka.adachi@tmu.ac.jp>}

\date{\today}



\begin{document}

\maketitle

\begin{abstract}
Classical filtrations in probability theory formalize the accumulation of information 
along a linear time axis: 
the past is unique and the present evolves into an uncertain future. 
In reality, however, 
this linearity may itself be an illusion
— an artifact of human perception 
that collapses multiple possible histories into a single apparent path.

In this paper, 
we propose a geometric and homological model of synthetic filtrations, 
where the present arises as a synthesis of many potential pasts. 
To achieve this, 
we introduce a new category $\bSigma$, 
extending the simplex category $\bDelta$ 
so that each moment of time carries contextual structure. 
Synthetic filtrations are realized as contravariant functors 
$\bSigma^{op} \to \Prob$, 
where $\Prob$ is the category of probability spaces with null-preserving maps.

We then develop a homological analysis of $\bSigma$-filtrations, 
constructing chain complexes whose boundaries are given by conditional expectations. 
Their homology groups measure informational ``holes’’ 
— probabilistic obstructions arising from the incompatibility of contextual expectations.

As a concrete realization, 
we define Dirichlet filtrations, 
in which measures on simplices arise from Dirichlet distributions, 
reflecting both parameter and contextual uncertainty. 
Bayesian updating is then interpreted as a categorical 
transformation of a Dirichlet functor, 
revealing learning as a reconstruction of coherence across contexts.

This framework suggests that 
what appears as a linear temporal order is merely a projection of a higher contextual geometry.
It unifies categorical probability, homological algebra, and Bayesian reasoning, 
offering a new language for uncertainty in mathematics, finance, and cognition.
\end{abstract}


\section{Introduction}
\label{sec:introduction}
Filtrations are a cornerstone of probability theory, 
providing a formal description of how information accumulates over time. 
In the classical framework, 
time is modeled linearly: 
the past forms a unique path leading to the present, 
and the future extends from it. 
This view, while mathematically powerful, is rooted in a perceptual simplification. 
The idea that the past is a single, 
well-defined history may itself be an illusion — 
a projection created by the limitations of human cognition. 
Our sense of temporal coherence emerges from collapsing multiple possible histories 
into a single narrative, 
even though, in reality, the present may be shaped by several competing or overlapping contexts.

This observation motivates a reformulation of the concept of filtration. 
Instead of treating time as a linear index that unfolds uniquely, 
we consider a context-dependent notion of time, 
in which the present synthesizes information from multiple contextual pasts. 
To formalize this idea, 
we introduce a new category $\bSigma$, 
which extends the simplex category $\bDelta$ 
by incorporating contextual information into each time index. 
Objects of $\bSigma$ represent “times in context,” 
while morphisms encode transitions that preserve contextual compatibility.

A $\bSigma$-filtration is then defined as a contravariant functor
$
X : \bSigma^{op} \to \Prob 
$,
where $\Prob$ is the category of probability spaces with null-preserving maps. 
This categorical framework generalizes the traditional notion of filtrations 
$
T^{op} \to \Prob
$,
allowing us to model information flow over a branching, context-sensitive geometry 
rather than a single ordered chain. 
In this sense, 
stochastic processes become sections over a contextual time manifold, 
rather than trajectories along a one-dimensional timeline.
Beyond its categorical construction, 
a $\bSigma$-filtration admits a homological structure. 
Conditional expectation operators between probability spaces 
can be arranged into a chain complex, 
yielding homology groups 
that quantify probabilistic ``holes’’ — obstructions or redundancies 
in how expectations are coherently transmitted across contexts. 
This homological analysis of $\bSigma$-filtrations provides algebraic invariants 
that measure inconsistency or coherence in the evolution of information, 
analogous to topological cycles in geometry.

The geometric and homological aspects of our model 
jointly reveal a deeper view of uncertainty. 
Stochastic variability arises not only from randomness within each probability space 
but also from the multiplicity of possible contextual pasts. 
Thus, 
uncertainty itself acquires a topological dimension.

As a concrete realization, 
we introduce Dirichlet filtrations, 
where measures on simplices are induced by Dirichlet distributions. 
These filtrations capture both parameter uncertainty
— variability within probabilistic models —
and contextual uncertainty
— ambiguity over which contextual branch has led to the present. 
Furthermore, 
we show that Bayesian updating can be interpreted as a categorical 
update of a Dirichlet functor, 
representing learning as a process of restoring coherence across contexts.

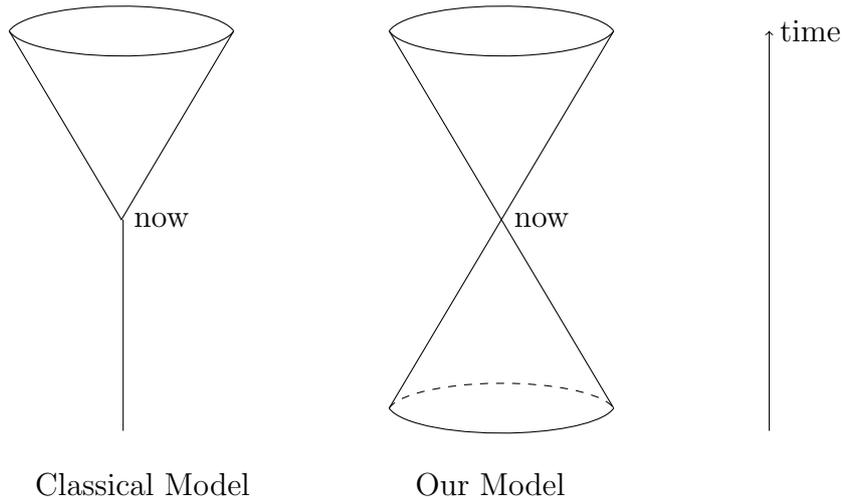
\begin{figure}
\begin{tikzpicture}
	\draw (0,0) arc (170:10:1.5cm and 0.4cm)coordinate[pos=0] (e);
	\draw (0,0) arc (-170:-10:1.5cm and 0.4cm)coordinate (f);
	\draw (e) -- ([yshift=-2.5cm]$(e)!0.5!(f)$) -- (f);
	\draw (1.5, -2.5) coordinate(g) node[right] {now}  -- (1.5, -5.3);

	\draw (0.2,-6) node[right] {Classical Model};

	\draw (5,0) arc (170:10:1.5cm and 0.4cm)coordinate[pos=0] (a);
	\draw (5,0) arc (-170:-10:1.5cm and 0.4cm)coordinate (b);
	\draw (a) -- ([yshift=-2.5cm]$(a)!0.5!(b)$) -- (b);

	\draw (6.5, -2.5) coordinate(g) node[right] {now};
	
	\draw[dashed] (5,-5) arc (170:10:1.5cm and 0.4cm)coordinate[pos=0] (c);
	\draw (5,-5) arc (-170:-10:1.5cm and 0.4cm)coordinate (d);
	\draw (c) -- ([yshift=2.5cm]$(c)!0.5!(d)$) -- (d);

	\draw (5.2,-6) node[right] {Our Model};

	\draw [->, very thin] (10, -5.3) -- (10,0) coordinate(g) node[right] {time};
\end{tikzpicture}

\label{fig:syntheticModel}
\caption{Synthetic model}
\end{figure}

The structure of the paper is as follows.
In section \ref{sec:Tfiltration},
we review categorical probability theory and 
generalized filtrations ($T$-filtrations) as a foundation.
In sectin \ref{sec:syncFilt},
we define the category $\bSigma$ of context-dependent time
and show how it extends the simplex category $\bDelta$.
We construct $\bSigma$-filtrations as contravariant functors 
$\bSigma^{op} \to \Prob$.
In section \ref{sec:homoSigma},
we develop a homological theory of $\bSigma$-filtrations, 
defining boundary operators from conditional expectations 
and interpreting their homology as algebraic measures of informational inconsistency.
In section \ref{sec:DirichletFilt},
we introduce Dirichlet filtrations, 
where probability measures on simplices arise from Dirichlet distributions, 
and interpret Bayesian updating as a categorical transformation of Dirichlet functors.

This geometric and homological perspective unifies ideas from category theory, 
simplicial geometry, homological algebra, 
and Bayesian probability. 
Beyond its conceptual novelty, 
it opens new directions for studying uncertainty 
in stochastic modeling, subjective probability, and mathematical finance.

\section{Review of $\mathcal{T}$-filtrations}
\label{sec:Tfiltration}

Before developing the context-dependent framework of $\bSigma$-filtrations, 
we first recall the foundations on which it is built. 
Classical filtrations are modeled as contravariant functors
from a time category to the category of probability spaces, 
and this idea extends naturally to the notion of a $\mathcal{T}$-filtration, 
where the time domain is an arbitrary small category $\mathcal{T}$.

In this section, 
we briefly review the categorical formulation of probability spaces 
together with conditional expectation operators defined on them,
and the definition of $\mathcal{T}$-filtrations 
introduced in 
\cite{AR_2019}
and
\cite{ANR_2020b}.
This review provides the structural background 
for the later generalization to $\bSigma$, 
which incorporates contextual dependence into the notion of time.

\subsection{Category of probability spaces}
\label{sec:catProb}

Let
$
X
	=
(\Omega_X, \mathcal{F}_X, \mathbb{P}_X)
$
and
$
Y
	=
(\Omega_Y, \mathcal{F}_Y, \mathbb{P}_Y)
$
be two probability spaces
throughout this section.

\begin{defn}{[Category $\Prob$]}
\label{defn:nullPres}
\begin{enumerate}
\item
A measurable map
$
f : 
(\Omega_X, \mathcal{F}_X)
	\to
(\Omega_Y, \mathcal{F}_Y)
$
is called
\newword{null-preserving}
if
\begin{equation}
\label{eq:nullPres}
\mathbb{P}_X \circ f^{-1}
	\ll
\mathbb{P}_Y
\quad
(\textrm{absolutely continuous}) .
\end{equation}
In other words,
it preserves negligible events under pushforward measures.

\item
$
\Prob
$
:
the category of probability spaces with null-preserving maps
\begin{itemize}
  \item
    $
    \mathcal{O}_{\Prob}
     :=
    $
    collection of all probability spaces,

  \item
    $
    \mathcal{M}_{\Prob}
     :=
    $
    collection of all null-preserving maps between probability spaces.
\end{itemize}

\end{enumerate}
\end{defn}

Note that any measure-preserving map is null-preserving.

\lspace

Other than $\Prob$,
in this paper,
we use the notations
$\Set$,
$\Cat$,
$\Top$,
$\Mble$,
$\Ord$
and
$\Ban$
for denoting categories of
sets,
small categories,
topological spaces,
measurable spaces,
preordered sets
and
Banach spaces,
respectively.
We also use
$
i_{\mathrm{Ord}}
	:
\Ord
	\to
\Cat
$
for the natural embedding functor,
and use
$
U_{\mathrm{PM}}
	:
\Prob
	\to
\Mble
$
for the forgetful functor.

\begin{defn}{[Functor $L^p$]}
\label{defn:functorL}
\begin{enumerate}
\item
A functor
$
\mathcal{L} : \Mble^{op} \to \Ban
$
is defined by
$
\mathcal{L}
	:=
\Mble(-, \mathcal{B}(\mathbb{R}))
$,
where
$\Ban$
is the category of Banach spaces,
$\mathbb{R}$ is considered as a (usual) topological space
of real numbers,
and
$   
\mathcal{B}
    :
\Top
    \to
\Mble
$
is the functor
that maps
every topological space
$(X, O_X)$
to the measurable space
$
\mathcal{B}((X, O_X))
	:=
(X, \mathcal{B}(O_X)) 
$.

\item
A functor
$
L
	:
\Prob^{op} \to \Ban
$
is defined by
\begin{equation*}
\xymatrix @C=30 pt @R=10 pt {
	\Prob^{op}
		\ar @{->}^{
			L
	} [rr]
&&
	\Ban
\\
	X
     \ar @{->}_{\varphi} [dd]
&&
	(\mathcal{L} \circ U_{PM})X / \sim_{\mathbb{P}_X}
		\ar @{}^{\ \ \ \ni} @<-6pt> [r]
&
	[g \circ \varphi]_{\sim_{\mathbb{P}_X}}
\\\\
	Y
&&
	(\mathcal{L} \circ U_{PM})Y / \sim_{\mathbb{P}_Y}
		\ar @{}^{\ \ \ \ni} @<-6pt> [r]
		\ar @{->}_{
			L(\varphi)
		} [uu]
&
	[g]_{\sim_{\mathbb{P}_Y}}
     \ar @{|->} [uu]
}
\end{equation*}
where 
$
\sim_{\mathbb{P}_X}
$
is an equivalence relation on 
$
	(\mathcal{L} \circ U_{PM})X
$
defined by
for
$
f, g \in
	(\mathcal{L} \circ U_{PM})X
$,
$
f
	\sim_{\mathbb{P}_X}
g
$
	iff
$
f = g
\quad
(\mathbb{P}_X\mathrm{-a.s.}).
$

\item
For 
$1 \le p < \infty$
A functor
$
L^{p}
	:
\Prob^{op} \to \Ban
$
is defined by
\begin{equation}
\label{eq:LpX}
L^{p} X
	:=
\{
	[f]_{\sim_{\mathbb{P}_X}}
		\in
	LX
\mid
	\norm{f}_p
		:=
	\Big(
		\int_{\Omega_X}
			\abs{f}^p
		d \mathbb{P}_X
	\Big)^{1/p}
		<
	\infty
\}.
\end{equation}

\item
A functor
$
L^{\infty}
	:
\Prob^{op} \to \Ban
$
is defined by
\begin{equation}
\label{eq:Linfty}
L^{\infty} X
	:=
\{
	[f]_{\sim_{\mathbb{P}_X}}
		\in
	LX
\mid
	\esssup_{\mathbb{P}_X} \abs{f}
		<
	\infty
\}.
\end{equation}

\end{enumerate}
\end{defn}

\begin{thm}{\normalfont{(\cite{AR_2019})}}
\label{thm:condExp}
Let
$\varphi : X \to Y$
be a null-preserving map.
For any random variable
$
\textblue{f}
	\in
L^{1}X
$,
there exists a
random variable
$g$
on
$Y$
such that
for every 
$
\textblue{B}
	\in
\mathcal{F}_Y
$,
\begin{equation}
\label{eq:condExp}
\int_B
	g \,
d \mathbb{P}_Y
	=
\int_{\varphi^{-1}(B)}
	f \,
d \mathbb{P}_X .
\end{equation}
We write 
$
\textblue{E^{\varphi}[f]}
$
for
the random variable
$g$,
and call it a
\newword{conditional expectation}
of
$f$
along
$\varphi$.
\end{thm}
\begin{proof}
Define a measure
$f^*$
on
$
(\Omega_X, \mathcal{F}_X)
$
as in the following diagram.
\begin{equation*}
\xymatrix@C=15 pt@R=20 pt{
&&
   A
       \ar @{|->}^{} [rr]
       \ar @{}_{ \mathrel{ \rotatebox[origin=c]{-90}{$\in$} } } @<+6pt> [d]
&&
   \textred{f^*}(A)
       \ar @{}^-{:=} @<-6pt> [r]
       \ar @{}_{ \mathrel{ \rotatebox[origin=c]{-90}{$\in$} } } @<+6pt> [d]
&
   \int_A f \, d \mathbb{P}_X
\\
   \mathcal{F}_Y
       \ar @{->}^{\varphi^{-1}} [rr]
       \ar @/_2pc/_{\mathbb{P}_Y} [rrrr]
&&
   \mathcal{F}_X
       \ar @{->}^{\textred{f^*}} @<+2pt> [rr]
       \ar @{->}_{\mathbb{P}_X} @<-2pt> [rr]
&&
   \mathbb{R}   
}
\end{equation*}
Then,
since
$
f^*
	\ll 
\mathbb{P}_X
$
and
$\varphi$
is null-preserving,
we have
\begin{equation*}
f^* \circ \varphi^{-1}
  \ll
\mathbb{P}_X \circ \varphi^{-1}
  \ll
\mathbb{P}_Y .
\end{equation*}
Therefore,
we get a following
Radon-Nikodym derivative.
\begin{equation*}
g
  :=
	d \,
  ( f^* \circ \varphi^{-1} )
/
	d \,
  \mathbb{P}_Y .
\end{equation*}
With this
$g$
we obtain for every
$B \in \mathcal{F}_Y$,
\begin{align*}
	&
\textblue{
 \int_B g \, d \mathbb{P}_Y
}
	=
\int_B \,
d (f^* \circ \varphi^{-1})
	=
(f^* \circ \varphi^{-1})(B)
 =
f^* ( \varphi^{-1}(B))
	=
\textblue{
 \int_{\varphi^{-1}(B)} f \, d \mathbb{P}_X 
} .
\end{align*}

\end{proof}

\begin{prop}{\normalfont{(\cite{AR_2019})}}
\label{prop:towerProp}
Let
$
X
	=
(\Omega_X, \mathcal{F}_X, \mathbb{P}_X)
$,
$
Y
	=
(\Omega_Y, \mathcal{F}_Y, \mathbb{P}_Y)
$
and
$
Z
	=
(\Omega_Z, \mathcal{F}_Z, \mathbb{P}_Z)
$
be three probability spaces,
and
$\varphi : X \to Y$
and
$\psi : Y \to Z$
be null-preserving maps.
Then, for
$
f \in L^{1}X
$, we have
\begin{equation}
\label{eq:towerProp}
E^{\psi}\big[
	E^{\varphi}[
		f
	]
\big]
		=
E^{\psi \circ \varphi}[
	f
]
	\quad
(\mathbb{P}_Z\mathrm{-a.s.}) .
\end{equation}

\end{prop}
\begin{proof}
All we need to show is that for every
$C \in \mathcal{F}_Z$,
$
\int_C 
E^{\psi}\big[
	E^{\varphi}[
		f
	]
\big]
\,
d \mathbb{P}_Z
		=
\int_C
E^{\psi \circ \varphi}[
	f
]
d \mathbb{P}_Z .
$
But, by iterating use of
(\ref{eq:condExp}),
we have
\[
\int_C 
E^{\psi}\big[
	E^{\varphi}[
		f
	]
\big]
\,
d \mathbb{P}_Z
		=
\int_{\psi^{-1}(C)}
	E^{\varphi}[f]
\,
d \mathbb{P}_Y
		=
\int_{\varphi^{-1}(\psi^{-1}(C))}
	f
\,
d \mathbb{P}_X
		=
\int_{(\psi \circ \varphi)^{-1}(C)}
	f
\,
d \mathbb{P}_X
		=
\int_C
E^{\psi \circ \varphi}[
	f
]
d \mathbb{P}_Z .
\]

\end{proof}

By Proposition \ref{prop:towerProp},
the following functor 
$\mathcal{E}$
is well-defined.

\begin{defn}{[Functor $\mathcal{E}$]}
\label{defn:functorE}
A functor
$
\mathcal{E}
	:
\Prob \to \Ban
$
is defined by
\begin{equation*}
\xymatrix @C=30 pt @R=10 pt {
	\Prob
		\ar @{->}^{
			\mathcal{E}
	} [rr]
&&
	\Ban
\\
	X
     \ar @{->}_{\varphi} [dd]
&&
	L^{1}X
		\ar @{}^{\ \ \ \ni} @<-6pt> [r]
		\ar @{->}^{
			\mathcal{E}(\varphi)
		} [dd]
&
	f
		\ar @{|->} [dd]
\\\\
	Y
&&
	L^{1}Y
		\ar @{}^{\ \ \ \ni} @<-6pt> [r]
&
	E^{\varphi}[f]
}
\end{equation*}
\end{defn}

We call
$\mathcal{E}$
the \newword{conditional expectation functor}.

\subsection{$\mathcal{T}$-filtration}
\label{sec:filtration}

Let
$\mathcal{T}$
be a small category which is considered as a \newword{time domain}
throughout this section.

\begin{defn}{[Filtration]}
\label{defn:filtration}
A $\mathcal{T}$-\newword{filtration} is a contravariant functor
$
F : 
\mathcal{T}^{op}
	\to
\Prob
$.
\end{defn}

\begin{figure*}[t]
\[
\xymatrix@C=30 pt@R=30 pt{
   \mathcal{T}^{op}
      \ar @{->}^{\textred{F}} [rrrr]
&&&&
   \Prob
\\
   s
      \ar @{->}_{*_{s,t}} [d]
      \ar @(dl,ul)[]^{1_s}
      \ar @/^2pc/^{
         *_{t,u} \circ *_{s,t}
      } [dd]
&&&&
	F(s)
		\ar @(dl,ul)[]^{
			\textblue{
			F(1_s) = 1_{F(s)}
			}
		}
		\ar @{}^-{:=} @<-6pt> [r]
&
	\textred{X_s}
\\
   t
      \ar @{->}_{*_{t,u}} [d]
      \ar @(dl,ul)[]^{1_t}
&&&&
   F(t)
      \ar @{->}^{\textred{\varphi_{t,s}} =: F(*_{s,t})} [u]
      \ar @(dl,ul)[]^{
        \textblue{
         F(1_t) = 1_{F(t)}
        }
      }
\\
   u
      \ar @(dl,ul)[]^{1_u}
&&&&
	F(u)
		\ar @{->}^{\textred{\varphi_{u,t}} =: F(*_{t,u})} [u]
		\ar @(dl,ul)[]^{
			\textblue{
			F(1_u) = 1_{F(u)}
			}
		}
		\ar @{}^-{:=} @<-6pt> [r]
		\ar @/_2pc/_{
			\textblue{
			F(*_{t,u} \circ *_{s,t}) = F(*_{s,t}) \circ F(*_{t,u}) = \textred{\varphi_{u,s}}
			}
		} [uu]
&
	\textred{X_u}
}
\]
\caption{$\mathcal{T}$-Filtration as a contravariant functor from $\mathcal{T}$ to $\Prob$}
\label{fig:filfun}
\end{figure*}

\begin{exmp}
\label{exmp:filteredProbSp}
Let
$
(\Omega, \mathcal{F}, \{ \mathcal{F}_t \}_{t \in \mathbb{R}_+}, \mathbb{P})
$
be a filtered probability space.
Then, we can define 
a filtration
$
F : \mathcal{T}^{op} \to \Prob
$
by
\begin{enumerate}
\item
$
\mathcal{T}
	:=
i_{Ord}(\mathbb{R}_+, \le)
$,

\item
For
$t \in \mathbb{R}_+$,
$
Ft
	:=
(\Omega, \mathcal{F}_t, \mathbb{P}|_{\mathcal{F}_t})
$,

\item
For $s, t \in \mathbb{R}_+$
with $s \le t$, define
$
F(*_{s,t}) :=  1_{\Omega} 
: Ft \to Fs .
$

\end{enumerate}
See Figure \ref{fig:tfilCla}.

\end{exmp}

\begin{figure*}[t]
\begin{equation*}
\xymatrix@C=20 pt@R=8 pt{
	\textblue{\Ord}
      \ar @{->}^{\textred{i_{\Ord}}} [rr]
&&
	\textblue{\Cat}
\\
	\mathbb{R}_+
      \ar @{|->}^{i_{\Ord}} [rr]
&&
   \textred{\mathcal{T}}^{\textblue{op}}
      \ar @{->}^{\textred{F}} [rrrr]
&&&&
   \Prob
\\
	s
		\ar @{}_{ \mathrel{ \rotatebox[origin=c]{90}{$\textblue{\ge}$} } } @<+6pt> [dd]
&&
   s
      \ar @(dr,ur)[]_{1_s}
      \ar @{->}^{\textred{*_{s,t}}} [dd]
      \ar @/_2pc/_{
         *_{s,u}
      } [dddd]
&&&&
	\textred{F(s)}
		\ar @(dl,ul)[]^{
			F(1_s) = 1_{F(s)}
		}
		\ar @{}^-{:=} @<-6pt> [r]
&
	\textblue{
		(\Omega, \mathcal{F}_s, \mathbb{P}|_{\mathcal{F}_s})
	}
&
	\mathcal{F}_s
      \ar @{->}^{\textblue{1_{\Omega}^{-1}}} [dd]
\\\\
	t
		\ar @{}_{ \mathrel{ \rotatebox[origin=c]{90}{$\textblue{\ge}$} } } @<+6pt> [dd]
&&
   t
      \ar @(dr,ur)[]_{1_t}
      \ar @{->}^{*_{t,u}} [dd]
&&&&
	\textred{F(t)}
      \ar @{->}^{\textblue{1_{\Omega}} =: \textred{F(*_{s,t})}} [uu]
      \ar @(dl,ul)[]^{
         F(1_t) = 1_{F(t)}
		}
		\ar @{}^-{:=} @<-6pt> [r]
&
	\textblue{
		(\Omega, \mathcal{F}_t, \mathbb{P}|_{\mathcal{F}_t})
	}
&
	\mathcal{F}_t
\\\\
	u
&&
	u
		\ar @(dr,ur)[]_{\textred{1_u} = *_{u,u}}
}
\end{equation*}
\caption{$\mathcal{T}$-filtration corresponding to a classical filtration}
\label{fig:tfilCla}
\end{figure*}
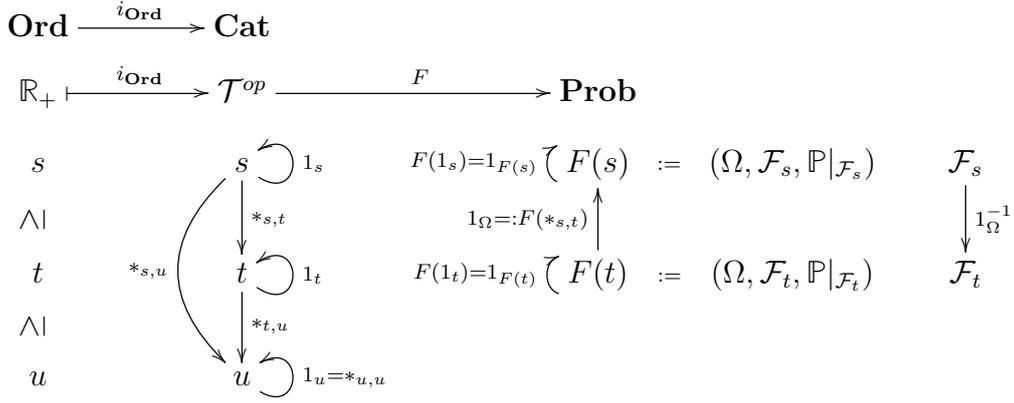

Example \ref{exmp:filteredProbSp}
shows that our filtration is a generalization of the classical filtration.

\section{Synthetic filtration}
\label{sec:syncFilt}

The central object of this paper is the notion of a \emph{synthetic filtration}, 
which generalizes classical filtrations by incorporating contextual information into the flow of time. 
In the standard setting, 
a filtration is a contravariant functor 
$F : {\mathcal{T}}^{op} \to \Prob$, 
where ${\mathcal{T}}$ is a linearly ordered index set such as $\mathbb{R}_+$ or $\mathbb{N}$. 
This structure presupposes that time unfolds along a single path, 
so that the present is determined by a unique past.

In contrast, 
our construction replaces the linear index category with a richer category $\bSigma$, 
in which each moment of time is equipped with a \emph{context}. 
Objects of $\bSigma$ are pairs $[t]_c$ consisting of a discrete time 
$t \in \mathbb{N}$ 
and a context 
$c \in \Sigma_0$, 
while morphisms are inherited from the simplex category $\bDelta$ 
subject to compatibility with the chosen context. 
Intuitively, $\bSigma$ allows us to encode the idea 
that the present state may be reached through multiple possible pasts, 
each corresponding to a different contextual path.

A \emph{$\bSigma$-filtration} is then defined as a contravariant functor
$
X : \bSigma^{op} \to \Prob
$,
which associates to each time-in-context $[t]_c$ a probability space and 
to each morphism the corresponding null-preserving map. 
This categorical formulation provides a geometric model 
in which uncertainty arises not only from stochastic variation but also from the multiplicity of contexts.

\subsection{Simplex category}
\label{simpCat}
\nocite{JT_2008}

In this and the next subsections, 
we review the simplex category,
which will be used as a fundamental structure when constructing 
category
$\bSigma$.
Please refer to textbooks such as
\cite{richter2020}
for the further details of the topic.

\begin{defn}{[Simplex category]}
\label{defn:simpCat}
\begin{enumerate}
\item
For $n \in \mathbb{N}$,
$
[n]
	:=
\{0, 1, \cdots, n\}
$
is a finite totally ordered set with the usual comparison of integers.

\item
$\bDelta$ is the full subcategory of $\Ord$ such that
\[
	\mathcal{O}_{\bDelta} := \{[n] | n \in \mathbb{N}\} .
\]
We call $\bDelta$ the \newword{simplex category}.

\item
For
$n \ge 1$
and
$i \in [n]$,
let
$
\delta^n_i
	:
[n-1] \to [n]
$
be the order-preserving inclusion
that misses the value $i$ in the target,
that is,
for
$k \in [n-1]$,
\begin{equation}
\label{eq:deltaNI}
\delta^n_i(k)
	:=
\begin{cases}
	k
		&\textrm{if }
	0 \le k \le i-1,
\\
	k+1
		&\textrm{if }
	i \le k \le n-1.
\end{cases}
\end{equation}

\item
For
$n \ge 0$
and
$j \in [n]$,
let
$
\sigma^n_j
	:
[n+1] \to [n]
$
be the order-preserving surjective map
that sends $j$ and $j+1$ to $j$,
that is,
for
$k \in [n+1]$,
\begin{equation}
\label{eq:sigmaNJ}
\sigma^n_j(k)
	:=
\begin{cases}
	k
		&\textrm{if }
	0 \le k \le j,
\\
	k-1
		&\textrm{if }
	j < k \le n+1.
\end{cases}
\end{equation}

\end{enumerate}
\end{defn}

\begin{exmp}
\label{exmp:exmpFaceDegen}
Let
$f : [3] \to [2]$
be an order-preserving map like the following:
\[
\xymatrix@C=60 pt@R=20 pt{
	[3]
		\ar @{->} [r]^{f}
&
	[2]
\\
	3
		\ar @{->} [rd]
&
\\
	2
		\ar @{->} [r]
&
	2
\\
	1
		\ar @{->} [rd]
&
	1
\\
	0
		\ar @{->} [r]
&
	0
}
\]
Then, $f$ can be represented like the following:
\[
\xymatrix@C=50 pt@R=20 pt{
	[3]
		\ar @{->} [r]^{\sigma^2_2}
&
	[2]
		\ar @{->} [r]^{\sigma^1_0}
&
	[1]
		\ar @{->} [r]^{\delta^2_1}
&
	[2]
\\
	3
		\ar @{->} [rd]
&
&
&
\\
	2
		\ar @{->} [r]
&
	2
		\ar @{->} [rd]
&
&
	2
\\
	1
		\ar @{->} [r]
&
	1
		\ar @{->} [rd]
&
	1
		\ar @{->} [ru]
&
	1
\\
	0
		\ar @{->} [r]
&
	0
		\ar @{->} [r]
&
	0
		\ar @{->} [r]
&
	0
}
\]
\end{exmp}

The next two propositions are well-known.

\begin{prop}
\label{prop:simpCatDeltaSigma}
Every arrow in the simplex category
$\bDelta$
can be represented as a composition of 
$\delta$s and $\sigma$s.
\end{prop}

\begin{prop}
\label{prop:faceDegenEqs}
We have the following equations.
\begin{align}
\label{eq:deltaDelta}
\delta^n_j
	\circ
\delta^{n-1}_i
		&=
\delta^n_i
	\circ
\delta^{n-1}_{j-1}
		\quad
(0 \le i < j \le n),
		\\
\label{eq:sigmaSigma}
\sigma^{n-1}_j
	\circ
\sigma^{n}_i
		&=
\sigma^{n-1}_i
	\circ
\sigma^{n}_{j+1}
		\quad
(0 \le i \le j \le n-1),
		\\
		\nonumber
\textrm{For }
n \ge 1,
i \in [n+1],
\textrm{ and }
j \in [n],
		\\
\label{eq:sigmaDelta}
\sigma^n_j
	\circ
\delta^{n+1}_i
		&=
\begin{cases}
\delta^{n}_i
	\circ
\sigma^{n-1}_{j-1}
		&\quad
(i < j),
			\\
1_{[n]}
		&\quad
(i = j \textrm{ or } i = j+1),
			\\
\delta^{n}_{i-1}
	\circ
\sigma^{n-1}_{j}
		&\quad
(i > j+1) .
\end{cases}
\end{align}

\end{prop}

\begin{defn}{[Simplicial objects]}
\label{defn:simplicialObj}
Let
$\mathcal{C}$
be a category.
A \newword{simplicial object}
in 
$\mathcal{C}$
is a contravariant functor
$
\bDelta^{op}
	\to
\mathcal{C}
$.
\end{defn}

\begin{Note}
\label{Note:simpObj}
Let
$
X
	:
\bDelta^{op} \to \mathcal{C}
$
be a simplicial object.
By Proposition \ref{prop:simpCatDeltaSigma},
$X$ is completely determined by specifying objects
\begin{equation}
\label{eq:Xn}
X_n := X[n]
\end{equation}
of $\mathcal{C}$ for every $n \in \mathbb{N}$
and
two sets of maps
\begin{equation}
\label{eq:faceMaps}
d^n_i := X(\delta^n_i)
	:
X_n \to X_{n-1}
	\quad
(n \ge 1 \; \textrm{and} \; 0 \le i \le n)
\end{equation}
called
\newword{face maps}
and
\begin{equation}
\label{eq:degeMaps}
s^n_j := X(\sigma^n_j)
	:
X_n \to X_{n+1}
	\quad
(n \ge 0 \; \textrm{and} \; 0 \le j \le n)
\end{equation}
called
\newword{degeneracy maps},
which satisfy the dual of 
equations in
Proposition \ref{prop:faceDegenEqs},
namely,
\begin{align}
\label{eq:SOdeltaDelta}
d^{n-1}_i
	\circ
d^n_j
		&=
d^{n-1}_{j-1}
	\circ
d^n_i
		\quad
(0 \le i < j \le n),
		\\
\label{eq:SOsigmaSigma}
s^{n}_i
	\circ
s^{n-1}_j
		&=
s^{n}_{j+1}
	\circ
s^{n-1}_i
		\quad
(0 \le i \le j \le n-1),
		\\
		\nonumber
\textrm{For }
n \ge 1,
i \in [n+1]
\textrm{ and }
j \in [n],
		\\
\label{eq:SOsigmaDelta}
d^{n+1}_i
	\circ
s^n_j
		&=
\begin{cases}
s^{n-1}_{j-1}
	\circ
d^{n}_i
		&\quad
(i < j),
			\\
1_{X_n}
		&\quad
(i = j \textrm{ or } i = j+1),
			\\
s^{n-1}_{j}
	\circ
d^{n}_{i-1}
		&\quad
(i > j+1) .
\end{cases}
\end{align}

\end{Note}

\subsection{Geometric realization}
\label{sec:geomRealization}

The construction of geometric realization is essential for our framework 
because it provides a concrete topological space on which probability measures can be defined. 
While simplicial objects capture combinatorial data about time and context, 
their geometric realizations translate this information into shapes (simplices) 
that admit natural probability measures, 
such as the uniform distribution or, more generally, Dirichlet distributions. 
In this way, 
geometric realization acts as a bridge between the categorical structure of time (encoded in $\bDelta$ or $\bSigma$) 
and the analytic structure of probability theory. 
Without this step, filtrations would remain purely combinatorial, 
whereas our goal is to endow them with probabilistic and geometric content suitable for modeling uncertainty.

\subsubsection{Functor $\Delta$}

\begin{defn}{[Geometric realization]}
\label{defn:geomReal}
A functor
$
\Delta :
	\bDelta \to \Top
$
is defined by
	\begin{itemize}
	\item
For 
$n \in \mathbb{N}$,
$\Delta_n$
is a topological space 
defined as
the convex closure of the standard basis
$\mathbf{e}^n_0, \mathbf{e}^n_1, \cdots, \mathbf{e}^n_n$ 
of
$\mathbb{R}^{n+1}$
with the Euclidean topology.
Thus,
\begin{equation}
\label{eq:stdTopNsymplex}
\Delta_n
	=
\big\{
	\sum_{i=0}^n w_i \mathbf{e}^n_i
\mid
	0 \le w_i \le 1, \,
	\sum_{i=0}^n w_i = 1
\big\} .
\end{equation}
Set
$
\Delta[n] := \Delta_n.
$

	\item
For
$
f : [n] \to [m]
$
in $\bDelta$,
the map
$
\Delta f 
	: 
\Delta[n] \to \Delta[m]
$
is defined by
\begin{equation}
\label{eq:deltaFun}
\Delta f
\Big(
\sum_{i=0}^n
	w_i \mathbf{e}^n_i
\Big)
	:=
\sum_{i=0}^n
	w_i \mathbf{e}^m_{f(i)} .
\end{equation}

	\end{itemize}

\begin{figure}[]
\begin{center}
\begin{tikzpicture}
    
\coordinate (O) at (0,0) node at (O) [color=black, below] {$\Delta_0$};
\fill (O) circle (2pt);
    
\coordinate (A) at (2,0);
\coordinate (B) at (5,0);
\draw (A) -- (B);
\fill (A) circle (2pt);
\fill (B) circle (2pt);
\coordinate (AB) at (3.5,0) node at (AB) [below, color=black] {$\Delta_1$};

\coordinate (C) at (7,0);
\coordinate (D) at (10,0);
\coordinate (E) at (8.5,2.5);
\fill (C) circle (2pt);
\fill (D) circle (2pt);
\fill (E) circle (2pt);
\draw (C) -- (D) -- (E) --cycle;
\coordinate (CD) at (8.5,0) node at (CD) [below, color=black] {$\Delta_2$};

\end{tikzpicture}
\end{center}
\label{fig:stdNsymplex}
\caption{Standard $n$-simplexes}
\end{figure}

\end{defn}

\subsubsection{Geometric realizers of $\bDelta$}
\label{sec:GeoReal}

In this subsection,
we will consider a contravariant functor
$
\mathcal{R}
	:
\bDelta^{op} \to \Top
$
instead of the covariant functor
$
\Delta
	:
\bDelta \to \Top
$,
while keeping their object maps are same, that is,
$
\mathcal{R}([n]) := \Delta([n]) = \Delta_n.
$
We call $\mathcal{R}$ a geometric realizer.
Actually, a geometric realizer is a simplicial space.

By Note \ref{Note:simpObj},
all we need to define 
for determining the arrow part of the functor $\mathcal{R}$
is to
specify 
face and degeneracy maps:
$ d^{n}_i : \Delta_n \to \Delta_{n-1} $
and
$ s^{n}_i : \Delta_n \to \Delta_{n+1} $.
We provide a
geometric realizer
$\mathcal{R}_0$.

\begin{defn}{[Face and degeneracy maps of $\mathcal{R}_0$]}
\label{defn:faceDegeMapsForRzero}
For
$i, j \in [n]$,
we define two functions
$ d^{n}_i : \Delta_n \to \Delta_{n-1} $
and
$ s^{n}_i : \Delta_n \to \Delta_{n+1} $
by
for
$
\mathbf{w}
	=
\sum_{k=0}^n
	w_k
	\mathbf{e}^n_k
	\in
\Delta_n
$,
\begin{align}
\label{eq:XdZero}
d^{n}_i(\mathbf{w})
	&:=
\begin{cases}
	\sum_{k=0}^{i-2}
		w_k \mathbf{e}^{n-1}_k
			+
	(w_{i-1} + w_i) \mathbf{e}^{n-1}_{i-1}
			+
	\sum_{k=i}^{n-1}
		w_{k+1} \mathbf{e}^{n-1}_k 
	& (\textrm{if } i > 0), 
				\\
	\sum_{k=0}^{n-2}
		w_{k+1} \mathbf{e}^{n-1}_k
			+
	(w_{n} + w_0) \mathbf{e}^{n-1}_{n-1}
	& (\textrm{if } i = 0), 
\end{cases}
			\\
\label{eq:XsZero}
s^{n}_j(\mathbf{w})
	&:=
	\sum_{k=0}^{j-1}
		w_{k} \mathbf{e}^{n+1}_k
			+
	\sum_{k=j+1}^{n+1}
		w_{k-1} \mathbf{e}^{n+1}_k   .
\end{align}

\end{defn}

The next proposition says that
$d^n_i$
and
$s^n_j$
defined in
Definition \ref{defn:faceDegeMapsForRzero},
satisfy
the conditions stated in 
Note \ref{Note:simpObj}.

\begin{prop}
\label{prop:faceDegeMapsForRzero}
Functions
$d^n_i$
and
$s^n_j$
satisfy the followings:
\begin{align}
\label{eq:XSOdeltaDeltaZero}
d^{n-1}_i
	\circ
d^n_j
		&=
d^{n-1}_{j-1}
	\circ
d^n_i 
		\quad
(0 \le i < j \le n),
		\\
\label{eq:XSOsigmaSigmaZero}
s^n_i
	\circ
s^{n-1}_j
		&=
s^n_{j+1}
	\circ
s^{n-1}_i
		\quad
(0 \le i < j \le n-1),
		\\
\nonumber
		\textrm{For }
n \ge 1,
	\,
i \in [n+1]
\textrm{ and }
j \in [n],
		\\
\label{eq:XSOsigmaDeltaZero}
d^{n+1}_i
	\circ
s^n_j
		&=
\begin{cases}
s^{n-1}_{j-1}
	\circ
d^{n}_i
		&\quad
(i < j),
			\\
1_{X_n}
		&\quad
(i = j \textrm{ or } i = j+1),
			\\
s^{n-1}_{j}
	\circ
d^{n}_{i-1}
		&\quad
(i > j+1) .
\end{cases}
\end{align}
\end{prop}

We call
$
\mathcal{R}_0 : \bDelta^{op} \to \Top
$
the \newword{standard geometric realizer}.

\subsection{$\bSigma$-filtration}
\label{sec:SigmaFilt}

This section explains the core innovation of this paper.
We introduce a category
$\bSigma$
which will be used as a time domain for our synthetic model.

\begin{defn}{[Time in a context]}
\label{defn:timeInContextu}
Let 
$
t 
	\in
\mathbb{N}
	=
\{0, 1, 2, \cdots\}
$.
\begin{enumerate}
\item
\[
\Sigma_0
	:=
\prod_{t \in \mathbb{N}}
	[t]
\]
We call a member
$
c \in \Sigma_0
$
a \newword{context}.

\item
A map
\[
p_t
	:
\Sigma_0 \to [t]
\]
is the projection.
For
$c \in \Sigma_0$
we write
$
c_t
$
for
$
p_t(c)
$.

\item
A binary relation 
$\sim_t$
on
$
\Sigma_0
$
is defined by
for
$c, d \in \Sigma_0$,
\[
c
	\sim_t
d
	\; \Leftrightarrow	\;
\forall
s \ge t
	\,.\,
c_s = d_s .
\]

\item
For
$c \in \Sigma_0$,
\[
\twc{t}{c}
	:=
(t, [c]_{\sim_{t+1}}),
\]
which is called a \newword{time $t$ in the context $c$}.

\item
\[
\Sigma
	:=
\{
	\twc{t}{c}
\mid
	t \in \mathbb{N},
	c \in \Sigma_0
\}.
\]
\end{enumerate}
\end{defn}

\begin{prop}
\label{prop:timeInContext}
Let
$
t
	\in
\mathbb{N}
$
and
$c, d \in \Sigma_0$.
\begin{enumerate}
\item
$
\twc{0}{c}
	=
(0, \Sigma_0)
$
	\quad
(does not depend on $c$),

\item
$
\twc{t}{c}
	=
\twc{t}{d}
$
if and only if
$
c \sim_{t+1} d .
$
\end{enumerate}
\end{prop}

The following concept will be used for encoding elements of 
$\Sigma_0$.

\begin{defn}{[Cantor expansion]}
\label{defn:cantorExp}
The
\newword{Cantor expansion}
(
\newword{factorial number system}
or
\newword{factorial base system}
)
of
$r
	\in
[0,1)
$
is a sequence
$
\mathrm{Cant}(
	r)
=
(
c_1,
c_2,
c_3,
\cdots
)
$
such that
$c_k \in \mathbb{N}, (0 \le c_k \le k)$
and
\begin{equation}
\label{eq:CantorExp}
r
	=
\sum_{k=1}^{n}
	\frac{c_k}{(k+1)!}
\end{equation}
where
$n$ is an integer or $\infty$.
\end{defn}

\begin{exmp}
Here are some examples of
the Cantor expansions $\{c_k\}$ of real numbers $r \in [0,1)$.
\begin{enumerate}
	\item
$
\mathrm{Cant}(
\frac{1}{9}
) =
(0, 0, 2, 3, 2) ,
$
  \item
$
\mathrm{Cant}(
e - 2
) =
(1, 1, 1, \cdots ) ,
$
  \item
$
\mathrm{Cant}(
\frac{3}{13}
) =
(0, 1, 1, 2, 4, 1, 0, 5, 5, 4, 2, 10) .
$

\end{enumerate}
\end{exmp}

Since
\[
r
	=
\sum_{k=1}^{n}
	\frac{c_k}{(k+1)!}
	=
\frac{1}{2}
\big(
	c_1
		+
	\frac{1}{3}
	\big(
		c_2
			+
		\frac{1}{4}
		\big(
			c_3
				+
			\cdots
		\big)
	\big)
\big),
\]
we can calculate a Cantor expansion of a real number
$r$
by the following algorithm:
\begin{lstlisting}
function CantorExpansion(real r) { # output sequence is stored in c[k]
	k := 0
	r[k+1] := r
	do {
		k := k+1
		Pick c[k] in {0, 1, ..., k }
			such that  c[k] <= (k+1) * r[k] < c[k] + 1
		r[k+1] := (k+1) * r[k] - c[k]
		break if r[k+1] = 0
	}
}
\end{lstlisting}

For a given encoded number
$r$,
we can recover its corresponding Candor Expansion by executing the above algorithm ``CantorExpansion''.
Note that in the execution, we will get a \textit{next} branch $c_k$ in the context $r$ one by one in the loop of the algorithm.
This is exactly same as the behavior of the person living in $\bSigma$ who is getting to know the next branch at time 
$\twc{k}{c}$.

The following proposition says that this algorithm will terminate if and only if 
$r$ is a rational number.

\begin{prop}
\label{prop:RatCantorExp}
Let 
$r
	\in [0, 1)
$
be a real number.
There exists a 
Cantor expansion of $r$ 
whose length is finite if and only if
$r$ is a rational number.
Actually, if 
$
r = \frac{m}{n}
$
with natural numbers
$m, n$
such that
$n > 0$ and $0 \le m < n$,
then
the length of the Cantor expansion of $r$ is less than $n$.
\end{prop}
\begin{proof}
Suppose $n$
has a prime factorization:
$
n = 
q_1^{p_1}
q_2^{p_2}
	\cdots
q_{\ell}^{p_{\ell}}
$,
where
$p_j \ge 1$.
Then, for each
$j = 1, \cdots, \ell$,
\[
q_j^{p_j}
	\, \mid \,
q_j \cdot
(2  q_j)
	\cdots
(p_j  q_j)
	\, \mid \,
(p_j  q_j)! .
\]
Since 
$q_j$
is a prime number,
$n$
devides
$
N!
$,
where
\[
N
	:=
\max_{1 \le j \le \ell}
	(p_j  q_j) .
\]
Now,
let $s$
be an integer satisfying
$
N! = sn
$.
Then,
we have
$
r
	=
\frac{m}{n}
	=
\frac{sm}{N!}
$,
which implies that the length of the Cantor expansion of $r$ is finite.

The remaining is to show  that
$
N < n
$.
We will show this by proving 
$
p_j q_j \le q_j^{p_j}
$
for all $j$.
But it is easily achieved by the induction on
$p_j$.

\end{proof}

Any context 
$c \in \Sigma_0$
can be encoded as a real number in
$[0,1)$
by the isomorphism:
\[
\mathrm{Cant}
	:
[0,1)
	\to
\Sigma_0 .
\]

By Proposition \ref{prop:RatCantorExp},
any finite context $c$
(i.e. $c_i = 0$ for all but finite $i$'s)
can be encoded uniquely by a rational number in
$[0, 1)$.

\lspace

Now, we define our time domain,
a category called
$\bSigma$.

\begin{defn}{[Category $\bSigma$]}
\label{defn:univSpV}
\begin{enumerate}

\item
$\bSigma$
is the category whose object set is 
$\Sigma$
and the arrow set is given by:
\begin{equation}
\label{eq:SigmaArrowSet}
\bSigma(
	\twc{t}{c} ,
	\twc{t'}{c'} 
)
	:=
\begin{cases}
\bDelta(
	[t],
	[t']
)
		& \textrm{if }
c \sim_{(t \lor t')+1} c' ,
		\\
\emptyset
		& \textrm{otherwise} .
\end{cases}
\end{equation}

\item
For
$
c \in \Sigma_0
$,
the map
$
\delta^t_{c,i}
	:
\twc{t-1}{c}
	\to
\twc{t}{c}
$
is defined by
\begin{equation}
\label{eq:deltaContextI}
\delta^t_{c,i}
	:=
\delta^t_{i}.
\end{equation}
Especially, we write
\begin{equation}
\label{eq:deltaContext}
\delta^t_c
	:=
\delta^t_{c,c_t} .
\end{equation}

\item
For
$
c \in \Sigma_0
$,
the map
$
\sigma^t_{c,j}
	:
\twc{t+1}{c}
	\to
\twc{t}{c}
$
is defined by
\begin{equation}
\label{eq:sigmaContextJ}
\sigma^t_{c,j}
	:=
\sigma^t_{j}.
\end{equation}
Especially, we write
\begin{equation}
\label{eq:sigmaContext}
\sigma^t_c
	:=
\sigma^t_{c,c_t} .
\end{equation}

\item
The functor 
$
        U_{\Sigma\bDelta}
:
\bSigma \to \bDelta
$
is defined by 
for 
$t \in \mathbb{N}$
and
$c \in \Sigma_0$,
\[
\xymatrix@C=40 pt@R=10 pt{
	\bSigma
        \ar @{->}^{U_{\Sigma\Delta}} [r]
&
    \bDelta
\\
	\twc{t-1}{c}
        \ar @{->}_{f} [dd]
&
	[t-1]
        \ar @{->}_{f} [dd]
\\\\
	\twc{t}{c}
&
	[t]
}   
\]

\end{enumerate}
\end{defn}

We use a same notation of forgetful functor for this 
because it actually forgets context.


For example, a context
$c \in \Sigma_0$,
$
\twc{2}{c}
$
represents a tree with root 
$
[2]
	=
U_{\Sigma\bDelta} (\twc{2}{c})
$
like the following:
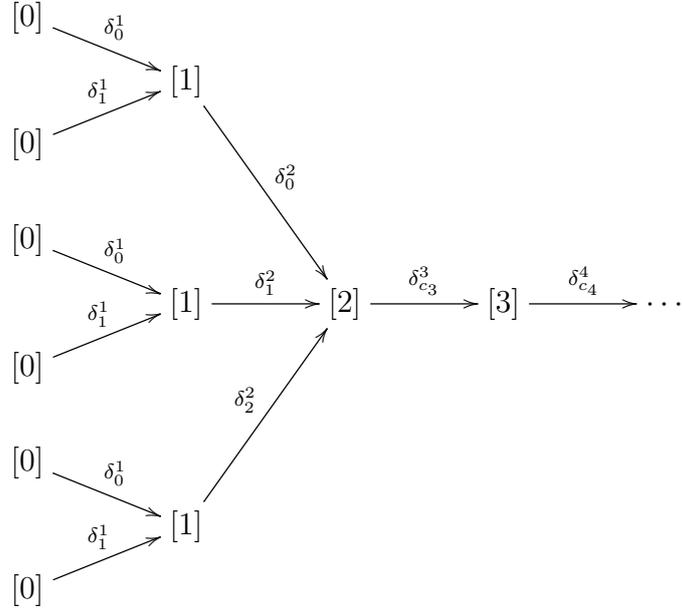
\[
\xymatrix@C=40 pt@R=5 pt{
	[0]
		\ar @{->}^{\delta^1_{0}} [rd]
\\
&
	[1]
		\ar @{->}^{\delta^2_{0}} [rdddd]
\\
	[0]
		\ar @{->}^{\delta^1_{1}} [ru]
\\
\\
	[0]
		\ar @{->}^{\delta^1_{0}} [rd]
\\
&
	[1]
		\ar @{->}^{\delta^2_{1}} [r]
&
	[2]
		\ar @{->}^{\delta^3_{c_3}} [r]
&
	[3]
		\ar @{->}^{\delta^4_{c_4}} [r]
&
	\cdots
\\
	[0]
		\ar @{->}^{\delta^1_{1}} [ru]
\\
\\
	[0]
		\ar @{->}^{\delta^1_{0}} [rd]
\\
&
	[1]
		\ar @{->}^{\delta^2_{2}} [ruuuu]
\\
	[0]
		\ar @{->}^{\delta^1_{1}} [ru]
}   
\]

\begin{figure*}[tb] 

\begin{tikzpicture}
	\draw [->, very thin] (1, -5.3) -- (1,0) coordinate(g) node[right] {time};



	\draw[dashed] (4,-4.5) arc (170:10:1.5cm and 0.4cm)coordinate[pos=0] (e);
	\draw (4,-4.5) arc (-170:-10:1.5cm and 0.4cm)coordinate (f);
	\draw (e) -- ([yshift=2.5cm]$(e)!0.5!(f)$) -- (f);

	\draw (5.5, -2.0) coordinate(g) node[right] {\textblue{tomorrow} $:= \twc{t+1}{c}$}  -- (5.5, 0);
	\node at (5.5, -2.0)[circle,fill,inner sep=1.5pt]{};

	\draw (1.1, -2.0) coordinate(g) node[right] {$t+1$};
	\node at (1, -2.0)[circle,fill,inner sep=1.5pt]{};

	\draw (5.5, 0) coordinate(g) node[right] {$c$}  -- (5.5, -2.5);


	\draw (5.8, -2.5) coordinate(g) node[right] {\textred{today} $:= \twc{t}{c}$};
	\node at (5.5, -2.5)[circle,fill,inner sep=1.5pt]{};

	\draw (1.1, -2.5) coordinate(g) node[right] {$t$};
	\node at (1, -2.5)[circle,fill,inner sep=1.5pt]{};
	
	\draw[dashed] (4,-5) arc (170:10:1.5cm and 0.4cm)coordinate[pos=0] (c);
	\draw (4,-5) arc (-170:-10:1.5cm and 0.4cm)coordinate (d);
	\draw (c) -- ([yshift=2.5cm]$(c)!0.5!(d)$) -- (d);



\end{tikzpicture}

\caption{Time evolution in $\bSigma$}

\label{fig:timeEvoSigma}
\end{figure*}
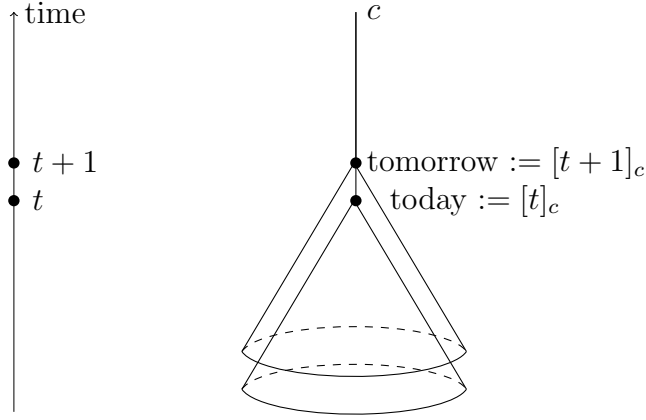

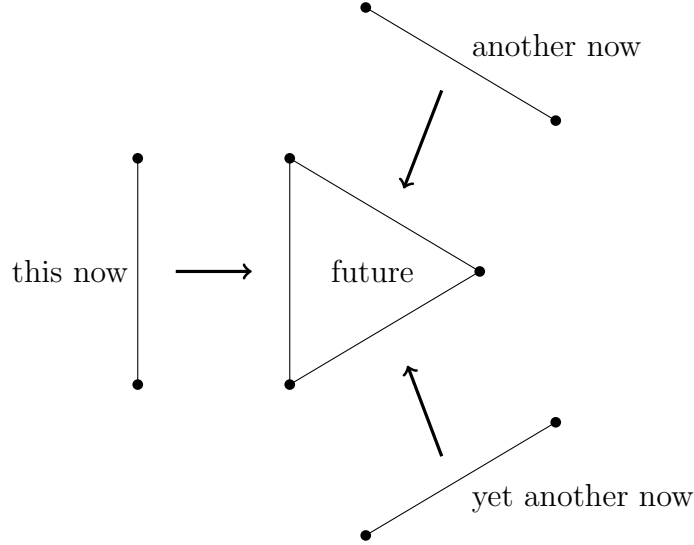
\begin{figure*}[tb] 
\begin{tikzpicture}

\coordinate (A) at (2,0);
\coordinate (B) at (2,3);
\coordinate (C) at (4.5,1.5);
\fill (A) circle (2pt);
\fill (B) circle (2pt);
\fill (C) circle (2pt);
\draw (A) -- (B) -- (C) --cycle;

\coordinate (A1) at (0,0);
\coordinate (B1) at (0,3);
\fill (A1) circle (2pt);
\fill (B1) circle (2pt);
\draw (A1) -- (B1);

\coordinate (B2) at (3, 5);
\coordinate (C2) at (5.5,3.5);
\fill (B2) circle (2pt);
\fill (C2) circle (2pt);
\draw (B2) -- (C2);

\coordinate (A3) at (3.0,-2);
\coordinate (C3) at (5.5,-0.5);
\fill (A3) circle (2pt);
\fill (C3) circle (2pt);
\draw (A3) -- (C3);

\draw[->, very thick, color=black] (0.5, 1.5) -- (1.5, 1.5);
\draw[->, very thick, color=black] (4, 3.9) -- (3.5, 2.6);
\draw[->, very thick, color=black] (4, -0.95) -- (3.55, 0.25);

\coordinate (AB) at (2.4,1.5) node at (AB) [right, color=black] {future};
\coordinate (AB1) at (-1.8, 1.5) node at (AB1) [right, color=black] {this now};
\coordinate (BC2) at (4.25, 4.5) node at (BC2) [right, color=black] {another now};
\coordinate (AC3) at (4.25, -1.5) node at (AC3) [right, color=black] {yet another now};

\end{tikzpicture}
\caption{Future is synthesized in many ``nows''}

\label{fig:manyNows}
\end{figure*}

In the rest of this paper,
we focus on a filtration
$
X : \bSigma^{op} \to \Prob
$
making the following diagram commute:
\[
\xymatrix@C=30 pt@R=30 pt{
	\bSigma^{op}
        \ar @{->}^{X} [rr]
        \ar @{->}_{U_{\Sigma \Delta}} [d]
&&
	\Prob
        \ar @{->}^{U_{PM}} [d]
\\
	\bDelta^{op}
        \ar @{->}^{\mathcal{R}_0} [r]
&
	\Top
        \ar @{->}^{\mathcal{B}} [r]
&
	\Mble
}   
\]
In other words,
the following equation holds:
\begin{equation}
\label{eq:commuXUPMeqBRUsigma}
U_{PM} \circ X
	=
\bar{\mathcal{R}} \circ U_{\Sigma\Delta} ,
\end{equation}
reminding that
\[
\bar{\mathcal{R}}
	:=
\mathcal{B} \circ \mathcal{R}_0 .
\]

Since
$
X
\delta^t_c
$
is null-preserving as an arrow of $\Prob$
for
$
\twc{t}{c}
	\in
\Sigma
$,
we have
\begin{equation}
\label{eq:muXdelta}
\mathbb{P}^c_{t-1}
	\circ
(d^t_{c,i})^{-1}
	\ll
\mathbb{P}^c_{t} ,
\end{equation}
where
$
\mathbb{P}^c_{t}
$
is a probability measure
of the probability space
$
X \twc{t}{c}
$,
that is,
\begin{equation}
\label{eq:XsingleTCsigma}
X \twc{t}{c}
	=
(
	\bar{\mathcal{R}}([t]),
	\mathbb{P}^c_{t}
)
	=
(
	\Delta_t,
	\mathcal{B}(\Delta_t),
	\mathbb{P}^c_{t}
)
\end{equation}
and
\[
d^t_{c,i}
	:=
X
\delta^t_{c,i}
	=
\bar{\mathcal{R}}
(
	U_{\Sigma \Delta}
	\delta^t_{c,i}
)
	=
\bar{\mathcal{R}}
\delta^t_{i} 
	=
d^t_{i} .
\]

We use the word
\newword{$\bSigma$-filtration}
only for denoting such filtration
$X$.

\section{Homological analysis of $\bSigma$-filtration}
\label{sec:homoSigma}

\nocite{hatcher2002}
\nocite{GQ2019}

The categorical construction of a $\bSigma$-filtration 
describes how probabilistic information flows through contextual time.
However, 
the structure of this flow itself possesses internal consistencies and obstructions, 
much like the topological structure of a simplicial complex.
To capture these features algebraically, 
we introduce a homological framework for $\bSigma$-filtrations.
Here, 
conditional expectations play the role of boundary operators, 
and their compositions reveal how information coheres — or fails to cohere — across contexts.

Let
$X$
be a $\bSigma$-filtration, i.e. 
$
X : \bSigma^{op} \to \Prob
$
throughout this section.

\subsection{Chain complex derived from $\bSigma$-filtration}
\label{sec:CCF}

\begin{defn}
\label{defn:HomProbTh}
Let
$
c \in \Sigma_0
$
be a context.
\begin{enumerate}

\item
$
C^c_t
	:=
L^1(X \twc{t}{c})
	=
L^1(\Delta_t, \mathcal{B}(\Delta_t), \mathbb{P}^c_t)
$.

\item
For $n \ge 1$,
a differential operator 
$
\partial^c_t
	:
C^c_t
	\to
C^c_{t-1}
$
is defined by:
\begin{equation}
\label{eq:DeltaDiffOp}
\partial^c_t
	:=
\sum_{i=0}^{t}
	(-1)^i
	\mathcal{E}(d^t_{c,i}) .
\end{equation}

\end{enumerate}
\end{defn}

Here, 
$
\mathcal{E}(d^t_{c,i})
$
acts like "averaging out" along the 
$i$-th direction.
Actually, 
by Theorem \ref{thm:condExp},
the differential operator
$\partial^c_t$
has the following property.

\begin{prop}
\label{prop:partialEq}
For
$f \in C_t$,
the random variable
$
\partial^c_t(f)
	\in
C^c_{t-1}
$
satisfies
for every
$B \in \mathcal{B}(\Delta_{n-1})$,
\begin{equation}
\label{eq:partialEq}
\int_B
	\partial^c_t(f)
d \mathbb{P}^c_{t-1}
		=
\sum_{i=0}^{t}
	(-1)^i
	\int_{
		(d^t_{c,i})^{-1}(B)
	}
		f
	d \, \mathbb{P}^c_t .
\end{equation}
\end{prop}

\begin{prop}
\label{prop:CompDiffOp}
For
$
t \ge 2
$,
we have
$
\partial^c_{t-1}
	\circ
\partial^c_t
	= 0
$.
\end{prop}
\begin{proof}
The proof has basically the same structure as those in classical simplicial homology.

First, for 
$f \in C^c_t$, we have
\begin{align*}
(\partial^c_{t-1} \circ \partial^c_t)(f)
			&=
\partial^c_{t-1}(
	\Big(
		\sum_{j=0}^{t}
			(-1)^j
			\mathcal{E}(d^t_{c,j})(f)
	\Big)
)
			\\&=
\sum_{i=0}^{t-1}
	(-1)^i
	\mathcal{E}(d^{t-1}_{c,i})
	\Big(
		\sum_{j=0}^{t}
			(-1)^j
			\mathcal{E}(d^t_{c,j})(f)
	\Big)
			\\&=
\sum_{i=0}^{t-1}
\sum_{j=0}^{t}
	(-1)^{i+j}
	\big(
		\mathcal{E}(d^{t-1}_{c,i})
			\circ
		\mathcal{E}(d^t_{c,j})
	\big)
	(f)
			\\&=
\sum_{i=0}^{t-1}
\sum_{j=0}^{t}
	(-1)^{i+j}
	\mathcal{E}(d^{t-1}_{c,i} \circ d^t_{c,j})
	(f).
\end{align*}
Thus,
\[
\partial^c_{t-1} \circ \partial^c_t
		=
\sum_{i=0}^{t-1}
\sum_{j=0}^{t}
	(-1)^{i+j}
	\mathcal{E}(d^{t-1}_{c,i} \circ d^t_{c,j}) .
\]

Now by
(\ref{eq:SOdeltaDelta}),
\[
d^{t-1}_{c,i} \circ d^t_{c,j}
	=
d^{t-1}_{c,j-1} \circ d^t_{c,i}
	\quad \textrm{for} \quad
i < j .
\]
Thus,
when we compute the double sum, for each 
$
(j, i)
$
with $j > i$,
we group terms corresponding to 
$(j, i)$
and
$(j-1, i)$,
and due to the alternating signs, these terms cancel out.
More precisely speaking,
let
$A := [n-1] \times [n]$,
$A_1 := \{ (i, j) \in A \mid i < j\}$
and
$A_2 := \{ (i, j) \in A \mid i \ge j\}$.
Then, 
$A_1 \cup A_2 = A$,
$A_1 \cap A_2 = \emptyset$,
while
the map
$
\varphi : A_1 \to A_2
$
defined by
$\varphi(i,j) := (j-1, i)$
is bijective.
Therefore, the above cancellation is well-done.
Hence,
$
\partial^c_{t-1}
	\circ
\partial^c_{t}
	= 0
$.

\end{proof}

By Proposition \ref{prop:CompDiffOp},
we have a chain complex per context:
\[
\xymatrix@C=20 pt@R=5 pt{
	0
&
	C^c_0
		\ar @{->}_{\partial^c_0} [l]
&
	C^c_1
		\ar @{->}_{\partial^c_1} [l]
&
	C^c_2
		\ar @{->}_{\partial^c_2} [l]
&
	\cdots
		\ar @{->}_{\partial^c_3} [l]
}
\]
where
the zero-map
$
\partial^c_0 : C^c_0 \to 0
$
is introduced for completeness, 
though originally, there are no boundary below degree $0$.

We are building a new kind of homology theory based on filtrations in probability.

\subsection{Homology of $\bSigma$-filtration}
\label{sec:HomoFilt}

Let
$
c \in \Sigma_0
$
be a context.
We move on to investigate the shape of homology group:
\begin{equation}
\label{eq:DhomologyGroup}
H^c_t := \ker(\partial^c_t) / \im(\partial^c_{t+1}),
\end{equation}
which is the quotient Abelian group consisting of subsets of 
$
C^c_t
	=
L^1(X \twc{t}{c})
$,
and is well-defined by the fact that
$
\im(\partial^c_{t+1})
	\subset
\ker(\partial^c_t)
	\subset
C^c_t
$
which is derived from
Proposition \ref{prop:CompDiffOp}.
We can think that $H^c_t$ measures the probabilistic ``holes'' or ``obstructions'' in stitching together conditional expectations over time.

Note that 
$ H^{c'}_t $
may not coincide with
$ H^c_t $
even if
$
c
	\sim_{t+1}
c'
$,
that is, if
$
\twc{t}{c}
	=
\twc{t}{c'}
$.
Therefore,
$H$
cannot be defined as a functor
over the category $\bSigma$.

Now,
let us compute
$H^c_0$, the zeroth homology group.
Since 
$
\partial^c_0 : C^c_0 \to 0
$,
this simplifies to:
\[
H^c_0
	=
C^c_0 / \im(\partial^c_1) .
\]
Here,
$
C^c_0
	=
L^{1}(X \twc{0}{c})
$
is the space of all integrable random variables at the initial time.
Since
$
\partial^c_1
	=
\mathcal{E}(d^1_{c,0})
	-
\mathcal{E}(d^1_{c,1})
$,
$
\im(\partial^c_1)
	\subset C^c_0
$
consists of functions that can be written as conditional expectations of functions from time $1$.
So,
$H^c_0$
measures the degree of expected deviation 
from the center due to the development to time $1$
along the context $c$.

Next let us compute
$
H^c_1
	=
\ker(\partial^c_1) / \im(\partial^c_2) 
$.
Here, we have
\[
\ker(\partial^c_1)
	=
\{
	f \in C^c_1
\mid
	\mathcal{E}(d^1_{c,0})(f)
		=
	\mathcal{E}(d^1_{c,1})(f)
\} .
\]
Therefore,
$H^c_1$
measures the space of ``$1$-cycles'' — functions at time $\twc{1}{c}$ whose conditional expectations at time $0$ cancel out
which are
like balanced observations from time $1$
— modulo those that come from $2$-step structure (boundary) of  $C^c_2$.

In the simplicial setting, the two face maps 
$d^1_{c,0}$
and
$d^1_{c,1}$
encode \textit{different past perspectives}, 
and
$\ker(\partial^c_1)$
encodes processes whose conditional expectations \textit{align} along all face maps
— like martingale consistency.

In general,
$H^c_t$
detects $t$-cycles or dependencies in the filtration 
that cannot be resolved by $(t+1)$-step behavior 
— a kind of coherence or anomaly in the conditional structure of the process over time.

Intuitionally, 
we can see
$H^c_0$
as
observable randomness at the present 
not explained by conditional expectations from the next time step,
$H^c_1$
as
balanced but non-trivial dependencies in the conditional structure 
— cycles in how expectations evolve,
and
for $t > 1$,
$H^c_t$
as
increasingly subtle “obstructions” 
to coherently reconstructing the filtration from expectations.

\section{Dirichlet filtration}
\label{sec:DirichletFilt}

To illustrate the abstract framework of 
$\bSigma$-filtrations with a concrete probabilistic model, 
we introduce the notion of a \emph{Dirichlet filtration}. 
The guiding idea is simple: 
simplices are natural carriers of probability measures, 
and the Dirichlet distribution provides a flexible family of such measures. 
Just as the Beta distribution plays a central role in one-dimensional Bayesian analysis, 
the Dirichlet distribution generalizes this role to higher-dimensional simplices, 
serving as the canonical conjugate prior for multinomial models. 
By assigning Dirichlet distributions to the simplices 
that arise in the geometric realization of 
$\bSigma$, 
we obtain a filtration that reflects both stochastic variability (through the distributional parameters) 
and contextual variability 
(through the choice of paths in $\bSigma$). 
This construction allows us to study uncertainty in a dual sense—uncertainty 
about future outcomes and 
uncertainty about which contextual past has given rise to the present.

\subsection{Dirichlet distribution}
\label{sec:DirichletDist}

This subsection is a review of
Dirichlet distributions on $n$-simplex.
Please refer to
\cite{ferguson_1973}
for the further details.

\begin{prop}
\label{prop:DeltaNvol}
Let
$\lambda_n$
be the Lebesgue measure on 
$\mathbb{R}^n$.
\begin{enumerate}
\item
$
\lambda_{n+1}(\Delta_n)
	=
\frac{1}{n!}
$,

\item
$
n! \lambda_{n+1}
$
is the uniform measure on 
$
\Delta_n
$.
\end{enumerate}
\end{prop}
\begin{proof}
For $n = 1, 2, 3, \cdots$, 
define functions
$
u_n : [0,1] \to \mathbb{R}
$
inductively by:
\[
u_1(t) := t,
		\quad
u_{n+1}(t)
	:=
\int_0^t
	u_n(t-x) 
	\,
d x ,
		\quad
	(t \in [0,1]).
\]
Then, we can easily show that
$
u_n(t) = \frac{t^n}{n!} ,
$
and so
$
\lambda_{n+1}(\Delta_n) =
u_n(1) = \frac{1}{n!}.
$

\end{proof}

The following definition of the gamma distribution is the version generalized by Ferguson
\cite{ferguson_1973}.

\begin{defn}{[Gamma distribution]}
\label{defn:gammaDist}
\begin{enumerate}
  \item
The \newword{gamma function}
$
\Gamma(\alpha)
$
is the function defined by for $\alpha > 0$,
\begin{equation}
\label{eq:GammaFunc}
\Gamma(\alpha)
	:=
\int_0^{\infty}
	x^{\alpha - 1}
	e^{-x} 
d x.
\end{equation}

\item
We denote by
$
\Gamma(\alpha, \lambda)
$
the \newword{gamma distribution} with shape parameter
$
\alpha \ge 0
$
and scale parameter
$
\lambda > 0
$.
For
$
\alpha = 0
$,
this distribution is degenerated at zero, that is, 
it becomes the Dirac delta function
$
\delta(x)
$;
for
$
\alpha > 0
$,
this distribution has density with respect to Lebesgue measure on the real line
\begin{equation}
\label{eq:gammaDist}
f_{
	\Gamma(\alpha, \lambda)
}
(x)
	:=
\begin{cases}
\frac{\lambda^{\alpha}}{\Gamma(\alpha)}
x^{\alpha - 1}
e^{- \lambda x}
		& \textrm{if }
x > 0,
		\\
0
		& \textrm{otherwise.}
\end{cases}
\end{equation}

\end{enumerate}
\end{defn}

Adopting the Dirac delta function when 
$\alpha = 0$
in Definition \ref{defn:gammaDist} (2)
is reasonable because
it is the limit distribution of
that of
any random variable
$
X \sim
\Gamma(\alpha, \lambda)
$
for
$\alpha > 0$,
whose expectation and variance are
$
\mathbb{E}[X]
	=
\frac{\alpha}{\lambda}
$
and
$
\var(X)
	=
\frac{\alpha}{\lambda^2}
$.

\begin{defn}
\label{defn:manyRs}
Let
$
n \ge 1
$
be an integer.
\begin{enumerate}
\item
$
\mathbb{R}_+^n
	:=
\{
	(x_0, \cdots, x_{n-1})
		\in
	\mathbb{R}^n
\mid
	x_i \ge 0
	\textrm{ for all $i \in [n-1]$}
\} ,
$

\item
$
\mathbb{R}_{\#}^n
	:=
\{
	(x_0, \cdots, x_{n-1})
		\in
	\mathbb{R}_+^n
\mid
	x_i > 0
	\textrm{ for some $i \in [n-1]$}
\} ,
$

\item
$
\mathbb{R}_{++}^n
	:=
\{
	(x_0, \cdots, x_{n-1})
		\in
	\mathbb{R}^n
\mid
	x_i > 0
	\textrm{ for all $i \in [n-1]$}
\} .
$
\end{enumerate}
\end{defn}

\begin{defn}{[Multivaliate beta function]}
\label{defn:multiBetaFunc}
Let 
$n \ge 1$
be an integer,
and
$
\boldsymbol{\alpha}
	=
(\alpha_0, \alpha_1, \cdots, \alpha_n) 
	\in
\mathbb{R}_{++}^{n+1}
$
be a vector.
\begin{enumerate}
  \item
For
$
\mathbf{x}
	=
(x_0, \cdots, x_n)
	\in
\Delta_n
$,
\begin{equation}
\label{eq:funcPiN}
\mathrm{b}_n(
	\mathbf{x}, 
	\boldsymbol{\alpha}
)
	:=
\prod_{i=0}^{n}
	x_i^{\alpha_i -1} .
\end{equation}

\item
$
\mathrm{B}_n
$
is the function, called \newword{the multivariate beta function}, 
defined by
\begin{equation}
\label{eq:multiBetaFun}
\mathrm{B}_n(
	\boldsymbol{\alpha}
)
	:=
\frac{
	\prod_{i=0}^{n}
		\Gamma(\alpha_i)
}{
	\Gamma\Big(
		\sum_{i=0}^n
			\alpha_i
	\Big)
}.
\end{equation}

\end{enumerate}
\end{defn}

The following fact is well-known.

\begin{prop}
\label{prop:multiBetaFun}
For
$
	\boldsymbol{\alpha}
=
	(\alpha_0, \cdots, \alpha_n)
\in
	\mathbb{R}_{++}^{n+1}
$,
\begin{equation}
\label{eq:multiBetaFunTwo}
\mathrm{B}_n(
	\boldsymbol{\alpha}
)
		=
\int_{\Delta_n}
	\mathrm{b}_n(
		\mathbf{x},
		\boldsymbol{\alpha}
	)
\, d \mathbf{x}.
\end{equation}
\end{prop}

\begin{defn}{[Dirichlet distribution]}
\label{defn:dirichDist}
Let 
$n \ge 1$
be an integer
and
let
$
\boldsymbol{\alpha}
	=
(\alpha_0, \alpha_1, \cdots, \alpha_n) 
	\in
\mathbb{R}_{\#}^{n+1}
$.
For
$j \in [n]$,
let
$
g_j
	\sim
\Gamma(\alpha_j, 1)
$
be independent random variables.
Then the random variable
$
\mathbf{f}
	=
(f_0, \cdots, f_n)
$
having values in
the standard $n$-simplex
$\Delta_n$,
defined by
\begin{equation}
\label{eq:dirichletRN}
f_i
	:=
\frac{g_i}{\sum_{j=0}^n g_j}
	\quad
(i \in [n]),
\end{equation}
is said to have \newword{the Dirichlet distribution} with parameter
$
\boldsymbol{\alpha}
$,
and is denoted by
$
\mathbf{f}
	\sim
\mathcal{D}_n(
		\boldsymbol{\alpha}
) .
$
Note that, if any
$\alpha_i = 0$,
the corresponding 
$f_i$
is degenerate at zero.
However, 
if
$
\boldsymbol{\alpha} \in \mathbb{R}_{++}^{n+1}
$,
it has a probability density function defined by
\begin{equation}
\label{eq:pnDirichret}
p_f(\mathbf{x})
	=
p_{\mathcal{D}_n}(\mathbf{x}; \boldsymbol{\alpha})
	:=
\frac{
	\mathrm{b}_n(
		\mathbf{x}, 
		\boldsymbol{\alpha}
	)
}{
	\mathrm{B}_n(
		\boldsymbol{\alpha}
	) .
}
\end{equation}
%

\end{defn}

Note that
$
p_{\mathcal{D}n}(x_0, \dots, x_n; 1, \cdots, 1) = n! ,
$
which is the uniform density over
$
\Delta_n
$.
Therefore,
the uniform distribution on $\Delta_n$ is a special case of Dirichlet distribution.

\begin{prop}
\label{prop:DirichletReprod}
Let
$
f = (f_0, \cdots, f_n)
	\sim
\mathcal{D}_n(
	\alpha_0, \cdots, \alpha_n
)
$
be a random variable,
$
0 < \ell \le n
$,
and
$
\{ r_j \}_{j=0}^{\ell}
$
be a sequence of integers such that
$
0 \le r_0 < r_1 < \cdots < \cdots r_{\ell -1} < r_{\ell} = n
$.
Then, we have

\begin{equation}
\label{eq::DirichletReprod}
\Big(
	\sum_{k=0}^{r_0}
		f_k,
	\sum_{k=r_0+1}^{r_1}
		f_k,
	\cdots
	\sum_{k=r_{\ell-1}+1}^{r_{\ell}}
		f_k
\Big)
		\sim
\mathcal{D}_{\ell}
\Big(
	\sum_{k=0}^{r_0} 
		\alpha_k,
	\sum_{k=r_0+1}^{r_1}
		\alpha_k,
	\cdots
	\sum_{k=r_{\ell-1}+1}^{r_{\ell}}
		\alpha_k
\Big) .
\end{equation}
\end{prop}
\begin{proof}
This follows directly from the definition of the Dirichlet distribution 
and the additive property of the gamma distribution:
$
f+g \sim \Gamma(a+b, \lambda)
$
for
two independent random variables
$
f \sim \Gamma(a, \lambda)
$
and
$
g \sim \Gamma(b, \lambda)
$.
\end{proof}

\begin{cor}
\label{cor:marginalDiri}
If
$
f \sim \mathcal{D}_n(
		\boldsymbol{\alpha}
)
$,
then
the marginal distribution of
$f_j$
is a $\beta$-distribution:
\[
f_j
	\sim
\beta\Big(\alpha_j,
	\sum_{i=0}^n \alpha_i
		- \alpha_j
\Big) .
\]
\end{cor}

\begin{Note}
\label{Note:conjugatePrior}
Let $n$ be a fixed integer.
Consider a random vector
$
\mathbf{r}
	=
(r_0, \cdots, r_n)
	\in
\mathbb{N}^{n+1}
$
following
a multinomial distribution
whose probability mass function is of the form:
\begin{equation}
\label{eq:multiNomialMassF}
p(\mathbf{r})
	=
\binom{N}{r_0 \cdots r_n}
q_0^{r_0}
q_1^{r_1}
	\cdots
q_n^{r_n}
	=
\binom{N}{r_0 \cdots r_n}
\mathrm{b}_n(
	\mathbf{q},
	\mathbf{r} + 1
),
\end{equation}
where
$
N
	:=
\sum_{i=0}^n r_i 
$
and
$
\mathbf{q}
	:=
(q_0, \cdots, q_n)
$.
Now let us assume that the prior distribution of
$
\mathbf{q}
	\in
\Delta_n
$
is a Dirichlet distribution, that is, it has of the form:
\begin{equation}
\label{eq:priorDistM}
p(\mathbf{q})
	=
\frac{
	\mathrm{b}_n(\mathbf{q}, 
		\boldsymbol{\alpha}
	)
}{
	\mathrm{B}_n(
		\boldsymbol{\alpha}
	)
} ,
\end{equation}
where 
$
		\boldsymbol{\alpha}
$
is a hyperparameter.

If we sample this random vector
$
\mathbf{r}
$, than we have
a likelihood function
with a parameter
$
\mathbf{x}
	\in
\Delta_n
$:
\begin{equation}
\label{eq:multiLikelifood}
\mathbb{P}(
	\mathbf{r}
\mid
	\mathbf{q}
		=
	\mathbf{x}
)
			=
\binom{N}{r_0 \cdots r_n}
\mathrm{b}_n(
	\mathbf{x},
	\mathbf{r} + 1
) .
\end{equation}
Then, since the prior distribution of
$\mathbf{q}$
is of the form:
\begin{equation}
\label{eq:qDistDiri}
\mathbf{P}(
	\mathbf{q}
		=
	\mathbf{x}
)
			=
\frac{
	\mathrm{b}_n(
		\mathbf{x},
		\boldsymbol{\alpha}
	)
}{
	\mathrm{B}_n(
		\boldsymbol{\alpha}
	)
}
			\; \sim \;
\mathcal{D}_n(
		\boldsymbol{\alpha}
) ,
\end{equation}
we have the following posterior distribution:
\begin{align}
\mathbf{P}(
	\mathbf{q}
		=
	\mathbf{x}
\mid
	\mathbf{r}
)
			&=
\frac{
	\mathbf{P}(
		\mathbf{r}
	\mid
		\mathbf{x}
	)
	\mathbf{P}(
		\mathbf{x}
	)
}{
	\int_{\Delta_n}
	\mathbf{P}(
		\mathbf{r}
	\mid
		\mathbf{y}
	)
	\mathbf{P}(
		\mathbf{y}
	)
	\,
	d \mathbf{y}
}
			=
\frac{
	\binom{N}{r_0 \cdots r_n}
	\mathrm{b}_n(
		\mathbf{x},
		\mathbf{r} + 1
	)
	\mathrm{b}_n(
		\mathbf{x},
		\boldsymbol{\alpha}
	)
		/
	\mathrm{B}_n(
		\boldsymbol{\alpha}
	)
}{
  \int_{\Delta_n}
	\binom{N}{r_0 \cdots r_n}
	\mathrm{b}_n(
		\mathbf{y},
		\mathbf{r} + 1
	)
	\mathrm{b}_n(
		\mathbf{y},
		\boldsymbol{\alpha}
	)
		/
	\mathrm{B}_n(
		\boldsymbol{\alpha}
	)
  \, d \mathbf{y}
}
			\nonumber
			\\&=
\frac{
	\mathrm{b}_n(
		\mathbf{x},
		\boldsymbol{\alpha}
			+
		\mathbf{r}
	)
}{
  \int_{\Delta_n}
	\mathrm{b}_n(
		\mathbf{y},
		\boldsymbol{\alpha}
			+
		\mathbf{r}
	)
  \, d \mathbf{y}
}
			\nonumber
			=
\frac{
	\mathrm{b}_n(
		\mathbf{x},
		\boldsymbol{\alpha}
			+
		\mathbf{r}
	)
}{
	\mathrm{B}_n(
		\boldsymbol{\alpha}
			+
		\mathbf{r}
	)
}
			\; \sim \;
\mathcal{D}_n(
		\boldsymbol{\alpha}
			+
		\mathbf{r}
) .
\end{align}
So, 
the Dirichlet distribution is a \newword{conjugate prior} of the multinomial distribution,
which is a well-known fact.

\end{Note}

\subsection{Dirichlet filtration}
\label{sec:dirichletFilt}

Let
$
X : \bSigma^{op} \to \Prob
$
be a 
$\bSigma$-filtration.
In this section,
we will assign a concrete probability measure 
$
\mathbb{P}^c_t
$
in the probability space
$
X
(\twc{t}{c})
	=
(\Delta_t, \mathcal{B}(\Delta_t), 
\mathbb{P}^c_t
)
$.

First, we define the probability measure on 
$
(\Delta_n, \mathcal{B}(\Delta_n)
$
induced by the Dirichlet distribution
$
\mathcal{D}_n(\boldsymbol{\alpha})
$.

\begin{defn}{[Probability measure induced by the Dirichlet distribution]}
\label{defn:indPMdiri}
Let
$
\boldsymbol{\alpha} 
	=
(\alpha_0, \cdots, \alpha_n)
	\in
\mathbb{R}_{\#}^{n+1}
$ 
be a parameter.
Define a subset 
$Z$
of
$[n]$
by
\[
Z
	:=
\{
	j \in [n]
\mid
	\alpha_j = 0
\} ,
\]
and let
$
m
	:=
n - \abs{Z} .
$
Then, there exists a map
$
\pi
	:
[m] \to [n]
$
such that
$
j_1 < j_2
$
implies
$
\pi[j_1] < \pi[j_2]
$
for every
$
j_1, j_2
	\in
[m]
$,
and
$
\pi([m])
	=
[n] \setminus Z .
$

For
$
A
	\in
\mathcal{B}(\Delta_n)
$,
define a set
$
\tilde{A}
	\in
\mathcal{B}(\Delta_m)
$
and a parameter
$
\tilde{
	\boldsymbol{\alpha}
}
	\in
\mathbb{R}_{++}^{m+1}
$
by
\begin{align*}
\tilde{A}
	&:=
\{
	(x_{\pi(0)}, \cdots, x_{\pi(m)})
\mid
	(x_0, \cdots, x_n) \in A
		\textrm{ and }
	\forall j \in Z \,.\, x_j = 0
\},
		\\
\tilde{
	\boldsymbol{\alpha}
}
	&:=
(
	\alpha_{\pi(0)}, 
	\cdots,
	\alpha_{\pi(m)} 
) .
\end{align*}
The probability measure
$
\mu_{
	\boldsymbol{\alpha}
}
$
on
$
(\Delta_n, \mathcal{B}(\Delta_n))
$
induced by the Dirichlet distribution
$
\mathcal{D}_n(
	\boldsymbol{\alpha}
)
$
is
defined by
for
$A \in \mathcal{B}(\Delta_n)$,
\begin{equation}
\label{eq:IndMeasByDiriDist}
\mu_{
	\boldsymbol{\alpha}
}
(A)
	:=
\int_{
	\tilde{A}
}
	p_{\mathcal{D}_m}(
		\mathbf{y};
		\tilde{
			\boldsymbol{\alpha}
		}
	)
\, d \mathbf{y}
	=
\int_{
	\tilde{A}
}
	\frac{
		\mathrm{b}_m(\mathbf{y}, 
			\tilde{\boldsymbol{\alpha}}
		)
	}{
		\mathrm{B}_m(
			\tilde{\boldsymbol{\alpha}}
		)
	} 
\, d \mathbf{y} .
\end{equation}

\end{defn}

\begin{defn}{[Category $\mathbf{Diri}$]}
\label{defn:catDiri}
\begin{enumerate}
\item
For
a pair
$
	(n, \boldsymbol{\alpha})
$
with
$
	\boldsymbol{\alpha}
		\in
	\mathbb{R}_{\#}^{n+1}
$,
we assign a probability space
$
\mathbf{p}
	(n, \boldsymbol{\alpha})
$
defined by
\begin{equation}
\label{eq:diriProbSp}
\mathbf{p}
(n, \boldsymbol{\alpha})
		:=
(
	\Delta_n, \mathcal{B}(\Delta_n),
	\mu_{
		\boldsymbol{\alpha}
	}
) .
\end{equation}

\item
$
\mathbf{Diri}
$
is a category
whose object set is:
\begin{equation}
\label{eq:catDiriObj}
\mathcal{O}_{\mathbf{Diri}}
	:=
\{
	(n, \boldsymbol{\alpha})
\mid
	\boldsymbol{\alpha}
		\in
	\mathbb{R}_{\#}^{n+1}
\}
\end{equation}
and
\begin{equation}
\label{eq:catDiriArrow}
\mathbf{Diri}
\big(
	(m, \boldsymbol{\alpha}),
	(n, \boldsymbol{\beta})
\big)
		:=
\Prob\big(
	\mathbf{p}(m, \boldsymbol{\alpha}),
	\mathbf{p}(n, \boldsymbol{\beta})
\big) .
\end{equation}

\item
The correspondence
$
\mathbf{p}
$
can be thought as a naturally-defined functor:
\[
\mathbf{p}
	:
\mathbf{Diri}
	\to
\Prob .
\]
\end{enumerate}
\end{defn}


We can extend the maps 
$d^n_i$
and
$s^n_j$
defined in
Definition \ref{defn:faceDegeMapsForRzero}
to
\[
d^n_i
	:
\mathbb{R}_+^{n+1}
	\to
\mathbb{R}_+^{n}
	\quad
\textrm{and}
	\quad
s^n_j
	:
\mathbb{R}_+^{n+1}
	\to
\mathbb{R}_+^{n+2},
\]
respectively
by allowing their arguments
$\mathbf{w}$
be any element of 
$
\mathbb{R}_+^{n+1}
$.
In the following proposition, we use these extended maps
$d^n_i$
and
$s^n_j$
for manipulating parameter
$
\boldsymbol{\alpha}
$.

\begin{prop}
\label{prop:DiriMP}
Let
$n \in \mathbb{N}$
and
$
\boldsymbol{\alpha}
	\in
\mathbb{R}_{\#}^{n+1}
$.
Then for
$
i, j \in [n]
$, 
we have
\begin{equation}
\label{eq:DiriMPd}
\mu_{\boldsymbol{\alpha}}
	\circ
\big(
	d^n_i
\big)^{-1}
		=
\mu_{d^n_i(\boldsymbol{\alpha})} 
		\quad
		(\mathrm{if} \; n \ge 1),
\end{equation}
and
\begin{equation}
\label{eq:DiriMPs}
\mu_{\boldsymbol{\alpha}}
	\circ
\big(
	s^n_j
\big)^{-1}
		=
\mu_{s^n_j(\boldsymbol{\alpha})} .
\end{equation}
That is,
both
$
d^n_i
	:
\Delta_{n}
	\to
\Delta_{n-1}
$
and
$
s^n_j
	:
\Delta_n
	\to
\Delta_{n+1}
$
are measure-preserving maps in
$\Prob$.
\end{prop}
\[
\xymatrix@C=40 pt@R=10 pt{
	[n-1]
        \ar @{->}^{\delta^n_{i}} [r]
&
	[n]
&
	[n+1]
        \ar @{->}_{\sigma^n_{j}} [l]
\\
	\Delta_{n-1}
&
	\Delta_{n}
        \ar @{->}_{d^n_{i}} [l]
        \ar @{->}^{s^n_{j}} [r]
&
	\Delta_{n+1}
\\
	\mathcal{B}(\Delta_{n-1})
        \ar @{->}^{(d^n_{i})^{-1}} [r]
        \ar @{->}_{\mu_{d^n_i(\boldsymbol{\alpha})}} [ddr]
&
	\mathcal{B}(\Delta_n)
        \ar @{->}_{\mu_{\boldsymbol{\alpha}}} [dd]
&
	\mathcal{B}(\Delta_{n+1})
        \ar @{->}_{(s^n_{j})^{-1}} [l]
        \ar @{->}^{\mu_{s^n_j(\boldsymbol{\alpha})}} [ddl]
\\\\
&
	[0,1]
}   
\]
\begin{proof}
Almost obvious by
Definition \ref{defn:faceDegeMapsForRzero}
and
Proposition \ref{eq::DirichletReprod}.
\end{proof}

\begin{defn}{[Dirichlet functor]}
\label{defn:diriFunc}
A \newword{Dirichlet functor}
$D$
is a functor
\[
D :
\bSigma^{op}
	\to
\mathbf{Diri}
\]
such that
\begin{enumerate}
\item
$
D 
\twc{t}{c}
$
is of the form
$
(t, \boldsymbol{\alpha})
$
with
$
\boldsymbol{\alpha}
	\in
\mathbb{R}_{\#}^{t+1}
$,

\item
For
$t \ge 1$,
$
D
\twc{t-1}{c}
	=
\big(
t-1, 
d_{c_t}^t(
	\boldsymbol{\alpha}
)
\big)
$
if
$
D
\twc{t}{c}
	=
(
	t, 
	\boldsymbol{\alpha}
)
$.

\item
The following diagram commutes:
\[
\xymatrix@C=30 pt@R=30 pt{
	\bSigma^{op}
        \ar @{->}^{D} [r]
        \ar @{->}_{U_{\Sigma \Delta}} [d]
&
	\mathbf{Diri}
        \ar @{->}^{\mathbf{p}} [r]
&
	\Prob
        \ar @{->}^{U_{PM}} [d]
\\
	\bDelta^{op}
        \ar @{->}^{\mathcal{R}_0} [r]
&
	\Top
        \ar @{->}^{\mathcal{B}} [r]
&
	\Mble
}   
\]

\end{enumerate}

\end{defn}

Note that
$
\tilde{D}
	:=
\mathbf{p}
	\circ
D
$
is a 
$\bSigma$-filtration.
We call this 
$
\tilde{D}
$
a \newword{Dirichlet filtration}.

\lspace

For a Dirichlet functor
$D$,
if we know the value of
$
D \twc{t}{c}
$
for some
$
\twc{t}{c}
	\in
\Sigma
$,
then 
any past
$
D \twc{s}{c'}
$,
where
$
s < t
$
and
$
c' \sim_{t+1} c
$,
is completely determined
by the condition mentioned in
Definition \ref{defn:diriFunc} (2).
However, future is not like this in general.

\begin{Note}
\label{Note:posFutPara}

Assume that
$
D(\twc{t}{c}
	=
(t, \boldsymbol{\alpha}) .
$
Then,
the set of possible 
$
D(\twc{t+1}{c})
$
is:
\begin{equation}
\label{eq:nextDt}
\textred{
\mathbf{next}_c
}
( 
	(t, \boldsymbol{\alpha})
)
	:=
\{
	(t+1, \boldsymbol{\beta})
\mid
	\boldsymbol{\beta}
		\in
	\mathbb{R}_{\#}^{t+2},
\;
	d^{t+1}_k(
		\boldsymbol{\beta}
	)
		=
	\boldsymbol{\alpha}
\} ,
\end{equation}
where
$
\textred{
k
}
	:=
c_{t+1} .
$
\[
\xymatrix@C=20 pt@R=20 pt{
	\bSigma^{op}
        \ar @{->}_{D} [d]
&
	\twc{t}{c}
        \ar @{->}^{\delta^{t+1}_{c}} [rr]
&&
	\twc{t+1}{c}
\\
	\mathbf{Diri}
&
	(t, \boldsymbol{\alpha})
&&
	(t+1, \boldsymbol{\beta})
        \ar @{->}_{d^{t+1}_{c}} [ll]
}
\]

So, if
$k > 0$,
$
\beta_{k-1}
$
and
$
\beta_{k}
$
can take any values as long as they satisfy the equation:
\[
\beta_{k-1}
	+
\beta_{k}
	=
\alpha_{k-1}.
\]

This fact can be considered as a source of uncertainty laid in future,
which we may call
\newword{parameter uncertainty}.

\end{Note}

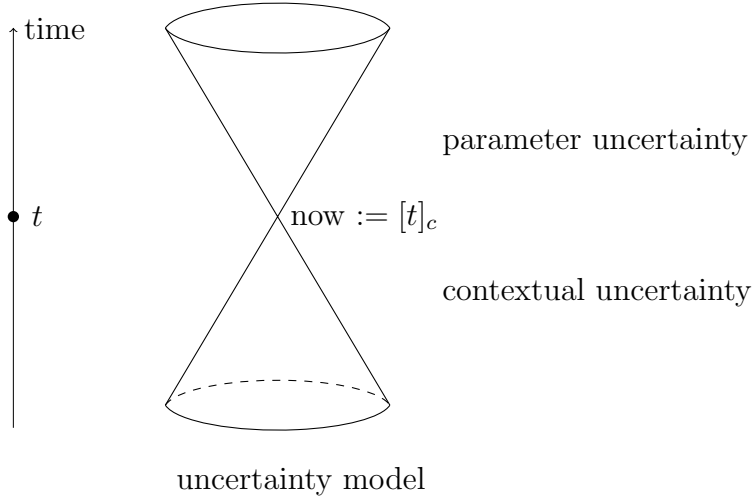
\begin{figure*}[tb]

\begin{tikzpicture}
	\draw [->, very thin] (1, -5.3) -- (1,0) coordinate(g) node[right] {time};

	\draw (1.1, -2.5) coordinate(g) node[right] {$t$};
	\node at (1, -2.5)[circle,fill,inner sep=1.5pt]{};



	\draw (3,0) arc (170:10:1.5cm and 0.4cm)coordinate[pos=0] (a);
	\draw (3,0) arc (-170:-10:1.5cm and 0.4cm)coordinate (b);
	\draw (a) -- ([yshift=-2.5cm]$(a)!0.5!(b)$) -- (b);

	\draw (4.5, -2.5) coordinate(g) node[right] {\textred{now} $:= \twc{t}{c}$};
	
	\draw[dashed] (3,-5) arc (170:10:1.5cm and 0.4cm)coordinate[pos=0] (c);
	\draw (3,-5) arc (-170:-10:1.5cm and 0.4cm)coordinate (d);
	\draw (c) -- ([yshift=2.5cm]$(c)!0.5!(d)$) -- (d);

	\draw (3.0,-6) node[right] {uncertainty model};

	\draw (6.5, -1.5) coordinate(g) node[right] {\textblue{parameter uncertainty}};
	\draw (6.5, -3.5) coordinate(g) node[right] {\textblue{contextual uncertainty}};

\end{tikzpicture}

\caption{Contextual uncertainty and parameter uncertainty}

\label{fig:contexualVsParameterUnvertainty}
\end{figure*}

\subsection{Bayesian consideration}
\label{sec:bayesianFun}
\nocite{bernardo_smith2000}


In this subsection,
we will see that
a Bayesian procedure can be seen as an update of Dirichlet functor.

Let
$
\textblue{D}
	:
\bSigma \to \mathbf{Diri}
$
be a fixed 
Dirichlet functor.

Suppose that
a \textblue{multinomial distribution}
$
\textred{
M_t
}
(\mathbf{q})
$
is assigned
to each time in a context
$
\textblue{
\twc{t}{c}
}
$,
whose probability distribution is determined by
for
$
\textred{
\mathbf{r}
}
	=
(r_0, \cdots, r_t)
	\in
\mathbb{N}^{t+1}
$,
\[
\textred{
p_{M_t(\mathbf{q})}
}
(\mathbf{r})
	:=
\binom{
	\sum_{i=0}^t r_i
}{
	r_0 \cdots r_t
}
\prod_{i=0}^t
	q_i^{r_i}
	=
\binom{
	\sum_{i=0}^t r_i
}{
	r_0 \cdots r_t
}
\mathbf{b}_n(
	\mathbf{q},
	\mathbf{r} + 1
)
\]
and
$
\textred{
\mathbf{q}
}
	=
(q_0, \cdots, q_t)
	\in
\Delta_t
$.

Moreover,
$
\mathbf{q}
$
itself
is a random variable having the \textblue{uniform distribution} in the probability space
\[
\tilde{D} \twc{t}{c}
	=
\mathbf{p}((t, 
\textred{
	\boldsymbol{\alpha}
}
	))
	=
(
	\Delta_t,
	\mathcal{B}(\Delta_t), 
	\textblue{
	\mu_{\boldsymbol{\alpha}}) .
	}
\]

\lspace

Now,
suppose that
we observe 
$
\textred{
k
}
	\in
[t]
$
at
$
\twc{t}{c}
$
as 
an \textblue{outcome}
of the multinomial distribution
$
M_t(\mathbf{q})
$.

Then,
the Dirichlet functor
$D$
will be updated to
another Dirichlet functor
$
\textred{
D'
}
$
in the following manner.
\begin{enumerate}
\item
At
$
\twc{t}{c}
$.
Since the Dirichlet distribution is a 
conjugate prior
of the multinomial distribution,
$
D'[t]_c
	:=
(t, 
	\boldsymbol{\alpha}'
)
$,
where
\[
\textred{
\boldsymbol{\alpha}'
}
	:=
(
\alpha_0, \cdots, \alpha_{k-1},
	\alpha_k \textblue{+ 1},
\alpha_{k+1}, \cdots, \alpha_t
).
\]

\item
For a past
$
\twc{s}{c'}
$
where
$s < t$
and
$
c' \sim_{t+1} c
$,
we can \textblue{uniquely} determine
$
D'[s]_{c'}
	=
(s, \textred{\boldsymbol{\beta}})
$
by iterating use of updating procedure of the parameter
with the face maps
$
d^t_{c}
$.

\item
For a future
$
\twc{u}{c}
$
where
$u > t$,
we can pick any value from the set calculated by
$\mathbf{next}_c$
in (\ref{eq:nextDt}),
which leads a \textblue{parameter uncertainty}.

\end{enumerate}

\begin{exmp}[Bayesian update on $\Delta_2$]

Let
$t = 2$.
Fix a context $c$ and consider time 
$\twc{t}{c}$ 
with state space 
$
\Delta_2=\{(x_0,x_1,x_2): x_i\ge 0,\ \sum_i x_i=1\}
$.
Let the prior on 
$\Delta_2$ 
be Dirichlet with parameter 
$
\boldsymbol{\alpha}
	=
(1,1,1)
$ 
(the uniform prior):
\[
q
	=
(q_0,q_1,q_2)
	\sim 
\mathrm{Dir}(1,1,1),
	\qquad 
\mathbb{E}[q_i]=\tfrac{1}{3}.
\]
At time $\twc{t}{c}$ 
we observe one multinomial trial with outcome ``category 1’’ (i.e., $k=1$). 
By conjugacy, the posterior becomes
\[
q\,|\,k=1\ \sim\ \mathrm{Dir}(1,\,1+1,\,1) = \mathrm{Dir}(1,2,1),
\]
so the updated expectation shifts to
\[
\mathbb{E}[q_0\,|\,k=1]=\tfrac{1}{4},\quad 
\mathbb{E}[q_1\,|\,k=1]=\tfrac{1}{2},\quad 
\mathbb{E}[q_2\,|\,k=1]=\tfrac{1}{4}.
\]

\emph{Filtration view.} 
Write 
$
D([t]_c)
	=
(2,
\boldsymbol{\alpha}
)
$
to indicate ``dimension $2$ with parameter 
$
\boldsymbol{\alpha}
$'' 
at time-in-context 
$\twc{t}{c}$. 
The Bayesian update defines a new Dirichlet functor $D'$ with
\[
D'([2]_c)
	=
(
	2,
	\boldsymbol{\alpha}'
)
	\quad\text{where}\quad
\boldsymbol{\alpha}'
	=
(1,2,1).
\]
For any past 
$\twc{s}{c'}$
 with 
$s < t$
and
$c'\sim_{t+1} c$, 
the value 
$D'(\twc{s}{c'})$ 
is determined functorially by composing with the appropriate face maps 
$d^t_c$
 (iterated as needed), 
i.e., by pushing the update backward along 
$\bSigma$’s arrows:
\[
D'(\twc{s}{c'})
	:=
\big(
	s, 
		\,
	d^{s+1}_{c'}(
		d^{s+2}_{c'}(
			\cdots
			d^{t}_{c'}(
				\boldsymbol{\alpha}'
			)
			\cdots
		)
	)
\big).
\]
Intuitively, 
the posterior information at $\twc{t}{c}$ restricts to coarser simplices (faces) in the past via the face maps.

\emph{Future uncertainty (parameter vs.\ context).} 
At 
$\twc{t+1}{c}$, 
the next-time parameter 
$
\boldsymbol{\beta}
$ 
must satisfy 
$
d^{\,t+1}_{k}(
	\boldsymbol{\beta}
)
	=
\boldsymbol{\alpha}'
$
where 
$
k
	=
c_{t+1}
$
(the next-time face designated by the context). 
If 
$
k \in \{1,2\}
$, 
this constraint only fixes the \emph{sum} of two adjacent coordinates of 
$
\boldsymbol{\beta}
$, 
leaving those two entries free to vary 
as long as they add to the corresponding entry of 
$
\boldsymbol{\alpha}'
$. 
Thus 
the future exhibits \emph{parameter uncertainty} even after we condition on the observed outcome; 
varying $c$ further reflects \emph{contextual uncertainty} about which face is selected.

\end{exmp}

\section{Conclusion}
\label{sec:conclusion}

In this work
we proposed a geometric and homological framework for \emph{synthetic filtrations},
extending the classical notion of filtrations in probability theory.
By introducing the category $\bSigma$,
which enriches the simplex category $\bDelta$
with context-dependent time,
we captured the idea that
the present may arise from the synthesis of multiple possible pasts.

The homological analysis of $\bSigma$-filtrations
revealed that conditional expectations can be organized into a chain complex,
leading to homology groups that quantify structural inconsistencies
or redundancies within probabilistic evolution.
These ``probabilistic homology classes’’
express how information fails to glue perfectly across contexts,
providing an algebraic measure of uncertainty beyond variance or entropy.

As a concrete realization,
we developed Dirichlet filtrations,
where probability measures on simplices are induced by Dirichlet distributions.
This construction highlighted two complementary sources of uncertainty:
parameter uncertainty, reflecting stochastic variability within models,
and contextual uncertainty, arising from the multiplicity of possible pasts.
Furthermore,
we showed that Bayesian updating can be formulated
as a categorical 
transformation of a Dirichlet functor,
offering a new interpretation of learning within our framework.

The theory of synthetic filtrations 
thus unites geometric, homological, and Bayesian viewpoints.
It opens several directions for further research. 
One avenue is the extension to continuous time 
and its relationship with classical stochastic processes
though there need some jumps to find appropriate continuous geometric structures. 
Another is 
the study of higher-dimensional homology invariants of $\bSigma$-filtrations,
potentially connecting with homotopy theory. 
Finally, applications in finance and economics, 
where information often arrives in context-sensitive and non-linear ways, 
represent a promising field of study.

Overall, 
synthetic filtrations provide a unifying geometric and categorical language 
for modeling time, uncertainty, and information, 
enriching both the theory of probability and its applications.


\section*{Acknowledgement}
I thank ChatGPT, a language model developed by OpenAI, 
for their assistance in exploring 
homological structures 
in probability theory,
I also thank Dr. Keisuke Hara of Association of Mathematical Finance Laboratory for his suggestion to use Beta distribution (or Dirichlet distribution) as a probability measure on the standard simplex.





\end{document}